\documentclass[11pt,a4paper]{amsart}

\usepackage{mymacros_english_paper}
\usepackage{mymacros_common}

\begin{document}
\title[Perfectoid towers generated from prisms]{Perfectoid towers generated from prisms}

\author[R. Ishizuka]{Ryo Ishizuka}
\address{Department of Mathematics, Tokyo Institute of Technology, 2-12-1 Ookayama, Meguro, Tokyo 152-8551}
\email{ishizuka.r.ac@m.titech.ac.jp}




\begin{abstract}
    We present a unified construction of perfectoid towers from specific prisms which covers all the previous constructions of (\(p\)-torsion-free) perfectoid towers.
    By virtue of the construction, perfectoid towers can be systematically constructed for a large class of rings with Frobenius lift.
    Especially, any Frobenius lifting of a reduced \(\setF_p\)-algebra has a perfectoid tower.
\end{abstract}

\maketitle 

\setcounter{tocdepth}{2}
\tableofcontents

\section{Introduction}

Let \(p\) be a prime number.
After Andr\'e \cite{andre2018Conjecture} and Bhatt \cite{bhatt2018Direct}, the theory of perfectoid rings has been applied to various fields of commutative algebra, especially in the study of mixed characteristic.
Recently, Ishiro--Nakazato--Shimomoto \cite{ishiro2025Perfectoid} developed the notion of \emph{perfectoid towers} (\Cref{PerfectoidTower}), which is a ``tower-theoretic'' generalization of perfectoid rings and gives an axiomatic Noetherian approximation of perfectoid rings.
The existence of perfectoid towers over a given (Noetherian local) ring is an extremely non-trivial problem and it is not known for general rings.
Bhatt--Scholze \cite{bhatt2022Prismsa} also generalized perfectoid rings by the notion of \emph{prisms} (\Cref{DefPrisms}), which is a ``deperfection'' of perfectoid rings and a building block of the theory of prismatic cohomology.

In this paper, we connect those two objects by constructing perfectoid towers from specific prisms.
All the previous constructions of (\(p\)-torsion-free) perfectoid towers can be unified by our construction.
Those examples are in \Cref{sec:examples}.
Our main theorem is the following:

\begin{theorem}[{Special case of \Cref{PerfectoidTowerPrism}}] \label{MainTheoremPrism}
    Let \((A, (d))\) be an (orientable) prism such that \(p, d\) is a regular sequence on \(A\) and \(A/pA\) is \(p\)-root closed\footnote{A ring \(A/pA\) is \emph{\(p\)-root closed in \(A/pA[1/d]\)} if \(x \in A/pA[1/d]\) satisfies \(x^{p^n} \in A/pA\) for some \(n \geq 1\), then \(x \in A/pA\) holds.} in \(A/pA[1/d]\).
    Then the tower of rings
    \begin{equation*}
        R_0 \defeq A/dA \xrightarrow{\varphi} A/\varphi(d)A \xrightarrow{\varphi} A/\varphi^2(d)A \xrightarrow{\varphi} \cdots \xrightarrow{\varphi} A/\varphi^i(d)A \xrightarrow{\varphi} \cdots
    \end{equation*}
    induced from the Frobenius lift \(\map{\varphi}{A}{A}\) becomes a perfectoid tower (\Cref{PerfectoidTower}) with injective transition maps.
    The \(p\)-completed colimit of the tower is isomorphic to the \(p\)-adic completion \((A_{\perf}/dA_{\perf})^{\wedge_p}\), where \(A_{\perf} \defeq \colim_{\varphi} A\).
    Furthermore, the tilt (\Cref{TiltPerfectoidTower}) of the perfectoid tower is isomorphic to
    \begin{equation*}
        A/pA \xrightarrow{F} A/pA \xrightarrow{F} A/pA \xrightarrow{F} \cdots,
    \end{equation*}
    where \(\map{F}{A/pA}{A/pA}\) is the Frobenius map.
\end{theorem}

This shows that, once we have such a prism, we can construct a perfectoid tower systematically and its tilt can be obtained immediately.
As a corollary, if a \(\delta\)-ring \(R\) satisfies some (weaker) conditions, then \(R\) has a perfectoid tower:

\begin{theorem}[{Special case of \Cref{ExplicitPerfectoidTowerDeltaRing}}] \label{MainTheoremDeltaRing}
    Let \(R\) be a \(p\)-torsion-free \(p\)-adically complete \(\delta\)-ring such that \(R/pR\) is reduced.
    Fix compatible sequences \(\{p^{1/p^i}\}_{i \geq 0}\) and \(\{\zeta_{p^i}\}_{i \geq 0}\) of \(p\)-power roots of \(p\) and unity in \(\overline{\setQ}\).
    Then the towers of rings
    \begin{align*}
        R \hookrightarrow R^{1/p} \otimes_{\setZ} \setZ[p^{1/p}] \hookrightarrow \cdots \hookrightarrow R^{1/p^i} \otimes_{\setZ} \setZ[p^{1/p^i}] \hookrightarrow \cdots, \, \text{and} \\
        R[\zeta_p] \hookrightarrow R^{1/p}[\zeta_{p^2}] \hookrightarrow \cdots \hookrightarrow R^{1/p^i}[\zeta_{p^{i+1}}] \hookrightarrow \cdots
    \end{align*}
    are perfectoid towers arising from \((R, (p))\) and \((R[\zeta_p], (p))\) respectively, where \(R^{1/p^i} \defeq \colim\{R \xrightarrow{\varphi} R \xrightarrow{\varphi} \cdots \xrightarrow{\varphi} R\}\) is the colimit consisting of \(i+1\) terms.

    The \(p\)-completed colimits of the towers are isomorphic to \((R_{\perf} \otimes_{\setZ} \setZ[p^{1/p^{\infty}}])^{\wedge_p}\) and \((R_{\perf} \otimes_{\setZ} \setZ[\zeta_{p^\infty}])^{\wedge_p}\) respectively, where \(R_{\perf} \defeq \colim_{\varphi} R\).
    Moreover, their tilts are both isomorphic to the tower
    \begin{equation*}
        R/pR[|T|] \xrightarrow{F} R/pR[|T|] \xrightarrow{F} R/pR[|T|] \xrightarrow{F} \cdots,
    \end{equation*}
    where \(R/pR[|T|]\) is the formal power series ring over \(R/pR\) with a variable \(T\).
\end{theorem}

To accomplish this theorem, we prove that the tensor product of a perfectoid tower \((\{A/\varphi^i(I)A\}, \{\varphi\})\) generated from a prism \((A, I)\) and a tower of rings \((\{R^{1/p^i}\}, \{\varphi_R\})\) generated from a \(\delta\)-ring \(R\) is again a perfectoid tower (\Cref{BaseChangePerfectoidTower}).
The degree of the generic extension of those transition maps are also determined in \Cref{GenericRankPerfectoidTower} in some cases.


This theorem says that (the \(p\)-adic completion of) any Frobenius lifting of a reduced Noetherian ring of characteristic \(p\) has a perfectoid tower (\Cref{PerfectoidTowerSingularity}).
Typical algebraic examples are the completion of any Stanley--Reisner ring \(\setZ_p[|\underline{T}|]/I_{\Delta}\) over \(\setZ_p\), i.e., quotients of square-free monomial ideals (\Cref{SquareFreeExample}) and the completion of any affine semigroup ring \(\setZ_p[|H|]\) over \(\setZ_p\) (\Cref{ExampleSemigroup}).
Moreover, geometrically, we can take a ring of sections \(R(\mcalX, \mcalL)\) of a canonical lift \(\mcalA\) of ordinary Abelian variety \(A\) and some ample line bundle \(\mcalL\) on \(\mcalA\).

In addition to these examples, we use a method of \(\delta\)-stabilization of ideals (\Cref{DeltaHeight}) and give a sufficient condition that the quotient of a formal power series ring by a \(\delta\)-stable ideal has a perfectoid tower in \Cref{ExampleDeltaStabilization}. This condition is in a form that can be determined by hand or computer algebra systems.

All previous examples of Noetherian perfectoid towers in \cite{ishiro2025Perfectoid} arise from regular local rings or local log-regular rings, which are Cohen-Macaulay and normal domains.
However, based on our results, the following examples of perfectoid towers are given:

\begin{example} \label{ExampleSummary}
    We can get perfectoid towers arising from a Noetherian ring which is
    \begin{itemize}
        \item a ramified complete intersection but not an integral domain (\Cref{SquareFreeExamplePrism}),
        \item a non-Cohen-Macaulay non-normal complete local domain (\Cref{ExampleSemigroupnonCM}),
        \item a non-regular complete intersection domain but not normal (\Cref{ExampleSemigroupnonNor}),
        \item a reduced non-Cohen-Macaulay and non-integral domain (\Cref{SquareFreeExample}),
        \item an unramified complete intersection domain but not log-regular with \(p=2\) (\Cref{ExampleDeltaStabilizationp2} following \Cref{p23CaseDeltaStabilization}), or
        \item a non-Cohen-Macaulay normal domain (\Cref{ExampleCone}).
    \end{itemize}
\end{example}

The first one is done by our first main theorem \Cref{MainTheoremPrism}. The next three examples are combinatorial examples such as affine semigroup rings and Stanley--Reisner rings. The fifth one is given by a (computer) calculation of \(\delta\)-stabilization of ideals. The last example is given by a ring of sections \(R(\mcalX, \mcalL)\) of geometric objects as explained above.

Very recently, Ishiro--Shimomoto \cite{ishiro2025drings} study a relation between perfectoid towers and lim Cohen-Macaulay sequences introduced by Bhatt--Hochster--Ma \cite{bhatt2024Lim} and give another way to construct perfectoid towers from certain \(\delta\)-rings. Our work is independent of theirs, but both works show that perfectoid towers can be constructed from various rings with, at least, Frobenius lift.
\addtocontents{toc}{\protect\setcounter{tocdepth}{-1}} 
\section*{Acknowledgements} 
\addtocontents{toc}{\protect\setcounter{tocdepth}{2}} 
The author would like to thank Shinnosuke Ishiro, Kei Nakazato and Kazuma Shimomoto for their valuable discussions about perfectoid towers, Shou Yoshikawa for his pointing out the usage of \(\phi\)-monomials, and Sora Miyashita for the many things he taught me about affine semigroup rings. Anonymous referees gave helpful comments and suggestions, for example adding more examples, examining the generic degrees of the transition maps, and pointing out some mistakes.
This work was supported by JSPS KAKENHI Grant Number 24KJ1085.

\section{Prisms}


In this section, we recall the notion of prisms and fix some terminology of towers of (\(\delta\)-)rings and prisms.


\begin{definition}[{\cite[Definition 2.1]{bhatt2022Prismsa}}] \label{DefDeltaRing}
    Let \(A\) be a ring.
    A \emph{\(\delta\)-structure on \(A\)} is a map of sets \(\map{\delta}{A}{A}\) such that \(\delta(0) = \delta(1) = 0\) and
    \begin{equation*}
        \delta(a + b) = \delta(a) + \delta(b) + \frac{a^p + b^p - (a + b)^p}{p}; \delta(ab) = a^p \delta(b) + b^p \delta(a) + p\delta(a)\delta(b)
    \end{equation*}
    for all \(a, b \in A\).
    A \emph{\(\delta\)-ring} is a pair \((A, \delta)\) of a ring \(A\) and a \(\delta\)-structure on \(A\).
    We often omit the \(\delta\)-structure \(\delta\) and simply say that \(A\) is a \(\delta\)-ring.
    An element \(d \in A\) is called a \emph{distinguished element} if \(\delta(d)\) is invertible in \(A\).

    On a \(\delta\)-ring \(A\), a map of sets \(\map{\varphi}{A}{A}\) is defined as
    \begin{equation*}
        \varphi(a) \defeq a^p + p \delta(a)
    \end{equation*}
    for all \(a \in A\).
    By the definition of \(\delta\), the map \(\varphi\) gives a ring endomorphism and we call it the \emph{Frobenius lift} on the \(\delta\)-ring \(A\).
    This induces the Frobenius map on \(A/pA\).

    We often use the symbol \(\varphi_*(-)\) and \(F_*(-)\) as the restriction of scalars along \(\varphi\) and \(F\), respectively.
\end{definition}

\begin{definition}[{\cite{bhatt2022Prismsa}}] \label{DefPrisms}
    A \emph{preprism} is a pair \((A, I)\) where \(A\) is a \(\delta\)-ring and \(I\) is an invertible ideal of \(A\).
    A preprism \((A, I)\) is a \emph{prism} (resp., \emph{Zariskian prism})\footnote{The terminology \emph{Zariskian prism} is non-standard and temporary, but we need to emphasize to only assume the Zariskian property instead of (derived) completeness.} if the following holds.
    \begin{enumerate}
        \item \(A\) is derived \((p, I)\)-complete (resp., \((p, I)\)-Zariskian).
        \item \(p \in I + \varphi(I)A\).
    \end{enumerate}

    A (pre)prism \((A, I)\) is called
    \begin{enumerate}
        \item \emph{perfect} if \(A\) is a perfect \(\delta\)-ring, i.e., \(\varphi\) is an isomorphism,
        \item \emph{bounded} if \(A/I\) has bounded \(p^\infty\)-torsion,
        \item \emph{orientable} if \(I\) is a principal ideal of \(A\),
        \item \emph{crystalline} if \(I = (p)\), or
        \item \emph{transversal} if \((A, I)\) is orientable and \(p, d\) is a regular sequence on \(A\) for some orientation \(d\) of \(I\) (or satisfies some equivalent conditions such as \cite[Lemma 2.9]{ishizuka2026Prismatic}). The transversal property was originally introduced in \cite{anschutz2020Pcompleted}.
    \end{enumerate}
\end{definition}

We will use the following simple example of Zariskian prisms in one of our main theorem \Cref{ExplicitPerfectoidTowerDeltaRing}.

\begin{example} \label{ExampleZariskianPrism}
    \begin{enumerate}
        \item Set a \(\delta\)-structure on \(\setZ_{(p)}[T]\) by \(\delta(T) = 0\). Then the pair \(((1 + (T))^{-1}\setZ_{(p)}[T], (p-T))\) is an orientable bounded Zariskian prism.
        \item The \(q\)-crystalline prism \((\setZ_p[|q-1|], ([p]_q))\) (\cite[Example 1.3 (4)]{bhatt2022Prismsa}) is an orientable bounded prism. Here \(\setZ_p[|q-1|]\) is the \((p, q-1)\)-adic completion of the polynomial ring \(\setZ[q]\) with the \(\delta\)-structure \(\delta(q) = 0\) and \([p]_q \in \setZ[q]\) is the \(q\)-analogue of \(p\), i.e., \([p]_q \defeq (q^p-1)/(q-1) = 1 + q + \cdots + q^{p-1}\).
        \item Similarly, the pair \(((1 + (q-1))^{-1}\setZ_{(p)}[q], ([p]_q))\) is an orientable bounded Zariskian prism.
    \end{enumerate}
    By easy observation, \(p, p-T\) (resp., \(p, [p]_q\)) is a regular sequence in \(\setZ_{(p)}[T]\) (resp., \(\setZ_{(p)}[q]\)) and \(\setF_p[T]\) (resp., \(\setF_p[q]\)) is \(p\)-root closed in \(\setF_p[T][1/T]\) (resp., \(\setF_p[q][1/[p]_q]\) because of \([p]_q \equiv (q-1)^{p-1}\) modulo \(p\)).
\end{example}


Next, we fix some terminology of towers and those morphisms.

\begin{definition} \label{DefTowers}
    \begin{enumerate}
        \item A \emph{tower of rings} is a sequence of rings \(A_0 \xrightarrow{\iota_0} A_1 \xrightarrow{\iota_1} A_2 \xrightarrow{\iota_2} \cdots\) where these \(\iota_i\) are maps of rings. We often denote \((\{A_i\}_{i \geq 0}, \{\iota_i\}_{\geq 0})\) as \((\{A_i\}, \{\iota_i\})\) (more simply, \((\{A_i\})\)).
        \item A \emph{tower of \(\delta\)-rings} is a tower of rings \((\{A_i\}_{i \geq 0}, \{\iota_i\}_{\geq 0})\) where each \(A_i\) is a \(\delta\)-ring and \(\iota_i\) is a map of \(\delta\)-rings.
        \item A \emph{tower of preprisms} is a pair \((\{A_i\}_{i \geq 0}, \{\iota_i\}_{\geq 0}, I)\) where \((\{A_i\}_{i \geq 0}, \{\iota_i\}_{\geq 0})\) is a tower of \(\delta\)-rings and \((A_0, I)\) is a preprism. We often denote \((\{A_i\}_{i \geq 0}, \{\iota_i\}_{\geq 0}, (A_0, I))\) as \((\{A_i\}, I)\).
        \item A \emph{tower of prisms} is a tower of preprisms \((\{A_i\}, I)\) such that each \(A_i\) is a derived \((p, I)\)-complete \(\delta\)-\(A\)-algebra and the preprism \((A_0, I)\) is a prism.
    \end{enumerate}
\end{definition}

\begin{remark} \label{RemarkTowerPrism}
    For any tower of (pre)prisms \((\{A_i\}, I)\), the base change \(I_{A_i} \defeq I \otimes^L_{A_0} A_i \to A_i\) gives an \emph{animated} (pre)prism \((I_i \to A_i)\) over the discrete prism \((A_0, I)\) by \cite[Corollary 2.10]{bhatt2022Prismatization}.
    So the tower of (pre)prisms gives a tower consisting of \emph{animated} (pre)prisms over \((A_0, I)\) whose underlying \(\delta\)-rings are discrete.

    Note that even if \(A_i\) are all discrete \(\delta\)-rings, each \((I_{A_i} \to A_i)\) is only an animated (pre)prism (see \cite[Remark 2.8]{bhatt2022Prismatization}).
    In the transversal case, this becomes an honest (pre)prism by \Cref{TransversalTowers}.
    The image of \(I_{A_i} \to A_i\) is an ideal \(IA_i\) of \(A_i\) generated by the image of \(I \subseteq A_0\) in \(A_i\).
\end{remark}

\begin{definition} \label{DefMapsTowers}
    Let \((\{A_i\})\) and \((\{A'_i\})\) be towers of rings.
    \begin{enumerate}
        \item A sequence of maps \(\map{f = (\{f_i\}_{i \geq 0})}{(\{A_i\})}{(\{A'_i\})}\) is a \emph{map of towers of rings} if \(f_i \colon A_i \to A'_i\) is a map of rings and compatible with \(\iota_i\) and \(\iota'_i\) for each \(i \geq 0\).
        \item If \((\{A_i\})\) and \((\{A_i\})\) are towers of \(\delta\)-rings, a map of towers of rings \(\map{f}{(\{A_i\})}{(\{A'_i\})}\) is a \emph{map of towers of \(\delta\)-rings} if each \(f_i\) is a map of \(\delta\)-rings.
        \item A \emph{map of towers of (pre)prisms} \(\map{f}{(\{A_i\}, I)}{(\{A'_i\}, I')}\) is a map of towers of \(\delta\)-rings \(\map{f}{(\{A_i\})}{(\{A'_i\})}\) such that \(\map{f_0}{A_0}{A'_0}\) induces a map of (pre)prisms \((A_0, I) \to (A_0', I')\).
    \end{enumerate}
\end{definition}

\section{Construction of Towers}
\label{sec:construction_of_towers}

In this section, we construct towers of rings from a given preprism (\Cref{PartiallyPerfection}) and the Frobenius projection of the tower (\Cref{FrobeniusProjection}).

\subsection{Construction of Towers from Prisms}

In positive characteristic, the Frobenius map \(F\) on a reduced ring \(R\) induces a tower of rings \(R \xrightarrow{F} R \xrightarrow{F} R \xrightarrow{F} \cdots\) which is called \emph{perfect tower}, and this is a perfectoid tower arising from \((R, pR)\) (\cite[Definition 3.2 and Example 3.23 (3)]{ishiro2025Perfectoid}).
Similarly, we can construct a ``perfect tower'' from a given (pre)prism, which is the main subject in this subsection.

\begin{construction} \label{ConstructionIsom}
    Let \((A, I)\) be a preprism. We denote \(\map{\varphi}{A}{A}\) a Frobenius lift as usual.
    We have a tower of \(\delta\)-rings
    \begin{equation*}
        A_0 \xrightarrow{\varphi} A_1 \xrightarrow{\varphi} A_2 \xrightarrow{\varphi} \cdots \xrightarrow{\varphi} A_i \xrightarrow{\varphi} \cdots,
    \end{equation*}
    where \(A_i\) is the same as \(A\) and \(\varphi\) is the Frobenius lift on \(A\).
    The map of rings \(\map{\varphi}{A_i}{A_{i+1}}\) induces the following maps of rings:
    \begin{align*}
        \overline{\varphi}_I^{(i)} \defeq \varphi/\varphi^i(I) \colon A_i/\varphi^i(I)A_i \to A_{i+1}/\varphi^{i+1}(I)A_{i+1}, \\
        \overline{\varphi}_{(p,I)}^{(i)} \defeq \varphi/(p, I^{[p^i]}) \colon A_i/(p, I^{[p^i]})A_i \to A_{i+1}/(p, I^{[p^{i+1}]})A_{i+1},
    \end{align*}
    where \(\varphi^i(I)A_i\) is the ideal of \(A_i\) generated by the image \(\varphi^i(I) \subseteq A_i\) of \(I \subseteq A_i\).
    Note that \(\overline{\varphi}_{(p, I)}^{(i)}\) is the \(p\)-th power map \(A_i/(p, I^{[p^i]})A_i \xrightarrow{a \mapsto a^p} A_{i+1}/(p, I^{[p^{i+1}]})A_{i+1}\).
    Those maps give three towers of rings:
    \begin{align*}
        (\{A_i\}, \{\varphi\}) & = A_0 \xrightarrow{\varphi} A_1 \xrightarrow{\varphi} A_2 \xrightarrow{\varphi} \cdots, \\
        (\{A_i/\varphi^i(I)A_i\}, \{\overline{\varphi}^{(i)}_I\}) & = A_0/\varphi(I)A_0 \xrightarrow{\overline{\varphi}_I^{(0)}} A_1/\varphi(I)A_1 \xrightarrow{\overline{\varphi}_I^{(1)}} A_2/\varphi^2(I)A_2 \xrightarrow{\overline{\varphi}_I^{(2)}} \cdots, \\
        (\{A_i/(p, I^{[p^i]})\}, \{\overline{\varphi}_{(p, I)}^{(i)}\}) & = A_0/(p, I)A_0 \xrightarrow{\overline{\varphi}_{(p,I)}^{(0)}} A_1/(p, \varphi(I))A_1 \xrightarrow{\overline{\varphi}_{(p,I)}^{(1)}} A_2/(p, \varphi^2(I))A_2 \xrightarrow{\overline{\varphi}_{(p,I)}^{(2)}} \cdots.
    \end{align*}
    The first tower becomes a tower of (pre)prism \((\{A_i\}, \{\varphi\}, I)\).
\end{construction}

As mentioned in \Cref{RemarkTowerPrism}, in the transversal case, the tower of prisms becomes a tower consisting of discrete prisms.

\begin{proposition} \label{TransversalTowers}
    If a prism \((A, I)\) is transversal in the sense of \Cref{DefPrisms}, then the animated prism \((I \otimes^L_{A, \varphi^i} A_i \to A_i) = (I \xrightarrow{\varphi^i} A_i)\) is a (discrete) orientable transversal prism \((A_i, \varphi^i(I)A_i)\).
\end{proposition}

\begin{proof}
    Fix an orientation \(d\) of \(I\).
    To prove the animated prism \((I \xrightarrow{\varphi^i} A_i)\) is a discrete prism, it suffices to show that \(\varphi^i(d)\) is a non-zero-divisor in \(A_i = A\) by \cite[Lemma 2.13]{bhatt2022Prismatization}, which follows from the transversal property of \((A, I)\).
    The transversal property of \((A_i, \varphi^i(I)A_i)\) also follows since \(p, \varphi^i(d)\) becomes a regular sequence in \(A\) for each \(i \geq 0\).
\end{proof}

Next, we show another representation of towers \((\{A_i\})\), \((\{A_i/\varphi^i(I)A_i\})\), and \((\{A_i/(p, I^{[p^i]})A_i\})\).
In concrete examples as in \Cref{sec:examples}, this representation is useful to understand the structure of the tower.

\begin{construction} \label{PartiallyPerfection}
    Let \((A, I)\) be a preprism.
    For each \(i \geq 0\), by \cite[Remark 2.7]{bhatt2022Prismsa} (or \cite[Proposition A.20 (1)]{bhatt2022Prismatization}), we can define a \(\delta\)-\(A\)-algebra \(A^{1/p^i}\) as the (finite) colimit in the category of \(\delta\)-rings
    \begin{equation*}
        A^{1/p^i} \defeq \colim \{A \xrightarrow{\varphi} A \xrightarrow{\varphi} \cdots \xrightarrow{\varphi} A\},
    \end{equation*}
    which is the colimit consisting of \((i+1)\)-copies of \(A\) with the Frobenius lift \(\varphi\).
    The underlying ring of \(A^{1/p^i}\) is the colimit of \(A\), that is, \(A^{1/p^i} \cong \colim \{A \xrightarrow{\varphi} A \xrightarrow{\varphi} \cdots \xrightarrow{\varphi} A\}\) as rings.
    In particular, if \(A\) is derived \((p, I)\)-complete, \(A^{1/p^i}\) is a derived \((p, I)\)-complete \(\delta\)-ring by \Cref{IsomorphicColimits} below.

    We denote the canonical map \(A \to A^{1/p^i}\) of \(\delta\)-rings from \(j\)-th term of the colimit as \(c_j^i\) for each \(0 \leq j \leq i\), namely, we have a map of \(\delta\)-rings
    \begin{equation*}
        c_j^i \colon A \to A^{1/p^i}.
    \end{equation*}
    In the following, \(A^{1/p^i}\) is a \(\delta\)-\(A\)-algebra via \(\map{c_0^i}{A}{A^{1/p^i}}\).
    Maps \(c_0^{i+1}, \dots, c_i^{i+1}\) of \(\delta\)-rings uniquely induce a map of \(\delta\)-\(A\)-algebras
    \begin{equation*}
        t_i \colon A^{1/p^i} \to A^{1/p^{i+1}}.
    \end{equation*}
    Set the ideal
    \begin{equation*}
        I_i \defeq c_0^i(I)A^{1/p^i} \subseteq A^{1/p^i},
    \end{equation*}
    then \(t_i\) induces a map of \(\delta\)-\(A\)-algebras
    \begin{align*}
        t_{i,I} \defeq t_i/I & \colon A^{1/p^i}/I_i \to A^{1/p^{i+1}}/I_{i+1}, \\
        t_{i,(p,I)} \defeq t_i/(p, I) & \colon A^{1/p^i}/(p, I_i) \to A^{1/p^{i+1}}/(p, I_{i+1}).
    \end{align*}
    Consequently, we have a tower of preprisms \((\{A^{1/p^i}\}, \{t_i\}, I)\) and two towers of rings \((\{A^{1/p^i}/I_i\}, \{t_{i,I}\})\) and \((\{A^{1/p^i}/(p, I_i)\}, \{t_{i,(p,I)}\})\).
    If \((A, I)\) is a prism, \((\{A^{1/p^i}\}, \{t_i\}, I)\) is a tower of prisms.

    By using commutativity of filtered colimits and tensor products, we have the following isomorphisms of \(A\)-algebras:
    \begin{align*}
        A^{1/p^i}/I_i & \cong \colim \{A/I \xrightarrow{\overline{\varphi}_I^{(0)}} A/\varphi(I)A \xrightarrow{\overline{\varphi}_I^{(1)}} \cdots \xrightarrow{\overline{\varphi}_I^{(i-1)}} A/\varphi^i(I)A\}, \\
        A^{1/p^i}/(p, I_i) & \cong \colim \{A/(p, I)A \xrightarrow{\overline{\varphi}_{(p,I)}^{(0)}} A/(p, \varphi(I))A \xrightarrow{\overline{\varphi}_{(p,I)}^{(1)}} \cdots \xrightarrow{\overline{\varphi}_{(p,I)}^{(i-1)}} A/(p, \varphi^i(I))A\} \\
        & \cong (A/pA)^{1/p^i}/I(A/pA)^{1/p^i}.
    \end{align*}

    As in \(\map{c_j^i}{A}{A^{1/p^i}}\), there exist the canonical maps of rings
    \begin{align*}
        c_{j, I}^i \colon A/\varphi^j(I)A \to A^{1/p^i}/I_i, \\
        c_{j, (p,I)}^i \colon A/(p, I^{[p^j]})A \to A^{1/p^i}/(p, I_i)
    \end{align*}
    for each \(i \geq 0\) and \(0 \leq j \leq i\).
\end{construction}

\begin{lemma} \label{IsomorphicColimits}
    Let \((A, I)\) be a preprism.
    Then the maps \(c_i^i\), \(c_{i, I}^i\) and \(c_{i,(p,I)}^i\) 
    \begin{align*}
        c_i^i & \colon A \xrightarrow{\cong} A^{1/p^i}, \\
        c_{i, I}^i & \colon A/\varphi^i(I)A \xrightarrow{\cong} A^{1/p^i}/I_i, \\
        c_{i,(p,I)}^i & \colon A/(p, I^{[p^i]})A \xrightarrow{\cong} A^{1/p^i}/(p, I_i)
    \end{align*}
    are isomorphisms of \(A\)-algebras, where the \(A\)-algebra structure on the left-hand sides are induced by the Frobenius lift \(\varphi\).
    In particular, if \(A\) is derived \((p, I)\)-complete, \(A^{1/p^i}\) (resp., \(A^{1/p^i}/I_i\)) is also derived \((p, I)\)-complete (resp., derived \(p\)-complete).
\end{lemma}

\begin{proof}
    Those isomorphisms follow from the definition of colimits.
    Since \(A\) (resp., \(A_i/\varphi^i(I)A_i\)) is derived \(\varphi^i(p, I)\)-complete (resp., derived \(p\)-complete), the completeness also follows.
\end{proof}

\begin{lemma} \label{IsomorphicTower}
    Let \((A, I)\) be a preprism.
    We can get isomorphisms of towers of rings between \Cref{ConstructionIsom} and \Cref{PartiallyPerfection} as follows.
    \begin{align*}
        \{c_i^i\} & \colon (\{A_i\},\{\varphi\}) \xrightarrow{\cong} (\{A^{1/p^i}\},\{t_i\}) \\
        \{c_{i, I}^i\} & \colon (\{A_i/\varphi^i(I)A_i\},\{\overline{\varphi}_I^{(i)}\}) \xrightarrow{\cong} (\{A^{1/p^i}/I_i\},\{t_{i, I}\}) \\
        \{c_{i,(p,I)}^i\} & \colon (\{A_i/(p, I^{[p^i]})A_i\},\{\overline{\varphi}_{(p,I)}^{(i)}\}) \xrightarrow{\cong} (\{A^{1/p^i}/(p, I_i)\},\{t_{i, (p, I)}\}).
    \end{align*}
    In particular, the first isomorphism \(\{c_i^i\}\) is an isomorphism of towers of preprisms between \((\{A_i\}, \{\varphi\}, I)\) and \((\{A^{1/p^i}\}, \{t_i\}, I)\).
\end{lemma}

\begin{proof}
    The first isomorphism is because the maps \(c_i^i\) are compatible with \(t_i\) by those constructions.
    The second and third isomorphisms follow from the same argument.
\end{proof}

By those isomorphisms, we have the following equivalences of injectivity of \(t_i\), \(t_{i, I}\), and \(t_{i, (p, I)}\).

\begin{corollary} \label{FrobeniusliftInjective}
    Let \((A, I)\) be a preprism and fix \(i \geq 0\).
    \begin{enumerate}
        \item The map of \(\delta\)-rings \(\map{t_i}{A^{1/p^i}}{A^{1/p^{i+1}}}\) is injective if and only if the Frobenius lift \(\map{\varphi}{A}{A}\) is injective. In this case, \(A\) is \(p\)-torsion free but the converse is not true in general (see \cite[Lemma 2.28]{bhatt2022Prismsa}).
        \item The map of rings \(\map{\overline{\varphi}_I^{(i)}}{A/\varphi^i(I)A}{A/\varphi^{i+1}(I)A}\) is injective if and only if the map \(\map{t_{i,I}}{A^{1/p^i}/I_i}{A^{1/p^{i+1}}/I_{i+1}}\) is injective.
        \item The \(p\)-th power map \(\map{\overline{\varphi}_{(p,I)}^{(i)}}{A/(p, I^{[p^i]})A}{A/(p, I^{[p^{i+1}]})A}\) is injective if and only if the map of rings \(\map{t_{i,(p,I)}}{A^{1/p^i}/(p, I_i)}{A^{1/p^{i+1}}/(p, I_{i+1})}\) is injective.
    \end{enumerate}
\end{corollary}

The injectivity of \(\varphi\) on \(A\) holds under some assumptions as follows.

\begin{lemma} \label{DeltaInjective}
    Let \(A\) be a \(p\)-adically separated \(\delta\)-ring.
    If \(A\) is \(p\)-torsion-free and \(A/pA\) is reduced, then the Frobenius lift \(\map{\varphi}{A}{A}\) is injective.
\end{lemma}

\begin{proof}
    If \(\varphi(x) = 0\) for \(x \in A\), then \(\overline{x}^p = 0\) in the reduced ring \(A/pA\) and thus there exists \(x_1 \in A\) such that \(x = px_1\). Since \(A\) is \(p\)-torsion-free, the equation \(0 = \varphi(x) = p \varphi(x_1)\) implies \(\varphi(x_1) = 0\) in \(A\). Repeating this argument, \(x\) is contained in \(\cap_{n \geq 0} p^nA = 0\). This shows the injectivity of \(\varphi\).
\end{proof}

The injectivity of \(\map{\overline{\varphi}_I^{(i)}}{A/\varphi^i(I)A}{A/\varphi^{i+1}(I)A}\) in \Cref{FrobeniusliftInjective} (2) follows under assumptions that are also assumed in our main theorem (\Cref{MainTheoremPrism}).

\begin{lemma} \label{InjectivityTower}
    Let \((A, (d))\) be an orientable preprism such that \(p, d\) is a regular sequence on \(A\) and \(A/pA\) is \(p\)-root closed in \(A/pA[1/d]\). Fix \(i \geq 0\).
    If \(A/\varphi^{i+1}(I)A\) is \(p\)-adically separated,\footnote{If \(A\) is derived \(p\)-complete, which is satisfied when \(A\) is a prism, then \(A/\varphi^i(I)A\) is \(p\)-adically separated for any \(i \geq 0\). This is a consequence that \(A/\varphi^i(I)A\) is derived \(p\)-complete and \(p\)-torsion-free.} then the map of rings \(\map{\overline{\varphi}_I^{(i)}}{A/\varphi^i(I)A}{A/\varphi^{i+1}(I)A}\) is injective.
\end{lemma}

\begin{proof}
    Since \(p, d\) is a regular sequence on \(A\), so is \(p, \varphi^{i+1}(d)\) and especially \(A/\varphi^{i+1}(d)A\) is \(p\)-torsion-free.
    Take an element \(a \in A\) such that \(\varphi(a) \in \varphi^{i+1}(d)A\). Taking modulo \(p\), this implies \(\overline{a}^p = \overline{d}^{p^{i+1}}A/pA\).
    The \(p\)-root closed assumption says that \(\overline{a} \in \overline{d}^{p^i}A/pA\) and there exist elements \(a_1, b_1 \in A\) such that \(a = pa_1 + \varphi^i(d)b_1 \in (p, \varphi^i(d))A\).
    Applying \(\varphi(-)\) to this equation, we have \(\varphi^{i+1}(d)A \ni \varphi(a) = p\varphi(a_1) + \varphi^{i+1}(d)\varphi(b_1)\) and thus \(\varphi(a_1) \in \varphi^{i+1}(d)A\) since \(A/\varphi^{i+1}(d)A\) is \(p\)-torsion-free.
    The same argument shows that there exist \(a_2, b_2 \in A\) such that \(a_1 = pa_2 + \varphi^i(d)b_2 \in (p, \varphi^i(d))A\) and thus \(a \in (p^2, \varphi^i(d))A\).
    Repeating this process, we have \(a \in (p^j, \varphi^i(d))A\) for all \(j \geq 0\) and thus \(\overline{a} \in \bigcap_{j \geq 0} p^jA/\varphi^i(d)A = 0\) by the \(p\)-adic separatedness of \(A/\varphi^i(d)A\). This shows the injectivity.
\end{proof}

    

\subsection{Construction of Frobenius Projections} \label{SectionConstFrobProj}

We next construct the ``Frobenius projection'' which plays a crucial role in the theory of perfectoid towers.
The existence of the Frobenius projection of a given tower is a key property in the theory of perfectoid towers (or \(p\)-purely inseparable tower as below).
The following observation and lemma (\Cref{FrobeniusProjectionFactor}) shows that the Frobenius projection of the tower \((\{A/\varphi^i(I)A\})\) generated from a prism \((A, I)\) is quite easily understood.

\begin{definition} \label{FrobeniusProjectionQuotient}
    Let \((A, I)\) be a preprism.
    Set the canonical surjection
    \begin{equation}
        \pi_i \colon A_{i+1}/(p, I^{[p^{i+1}]})A_{i+1} \twoheadrightarrow A_i/(p, I^{[p^i]})A_i
    \end{equation}
    induced from the identity map \(A_{i+1} \xrightarrow{\id} A_i\) of \(A = A_i = A_{i+1}\).
    We call the map \(\pi_i\) the \emph{\(i\)-th Frobenius projection} of the tower of prisms \((\{A_i\}, I)\) (or of the tower of rings \((\{A_i/\varphi^i(I)A_i\})\), see \Cref{PurelyInsepTowerPrism}).
\end{definition}

The following obvious lemma is a key observation in the theory of perfectoid towers.

\begin{lemma} \label{FrobeniusProjectionFactor}
    The Frobenius map \(F\) on \(A_{i+1}/(p, I^{[p^{i+1}]})A_{i+1}\) factors through \(\pi_i\) as follows:
    \begin{center}
        \begin{tikzcd}
            {A_{i+1}/(p, I^{[p^{i+1}]})A_{i+1}} \arrow[rr, "F"'] \arrow[rd, "\pi_i"', two heads] &                                                                            & {{A_{i+1}/(p, I^{[p^{i+1}]})A_{i+1}}.} \\
            & {A_i/(p, I^{[p^i]})A_i.} \arrow[ru, "{\overline{\varphi}_{(p, I)}^{(i)}}"'] &                                       
        \end{tikzcd}
    \end{center}
    Note that the map \(\map{\overline{\varphi}_{(p, I)}^{(i)}}{A_i/(p, I^{[p^i]})A_i}{A_{i+1}/(p, I^{[p^{i+1}]})A_{i+1}}\) is the \(p\)-th power map defined in \Cref{ConstructionIsom}.
\end{lemma}

The Frobenius projections of the tower of prisms \((\{A^{1/p^i}\})\) are also constructed as follows (\Cref{FrobeniusProjection}) and we record the compatibility under the isomorphisms \(\{c_i^i\}\) as in \Cref{IsomorphicTower} (\Cref{ObservationFrobeniusProj}).
However, this is a little bit complicated and thus the reader may skip to \Cref{PurelyInsepTower} and \Cref{PurelyInsepTowerPrism}.

\begin{construction} \label{FrobeniusProjection}
    Let \((A, I)\) be a preprism.
    Note that any ring \(A_i\) is \(A\) itself.
    Considering the following commutative diagram of rings;
    \begin{equation} \label{FrobeniusProjectionDiag}
        \begin{tikzcd}
            A_0 \arrow[d, "\varphi"] \arrow[r, "\varphi"]   & A_1 \arrow[d, "\varphi"] \arrow[r, "\varphi"]   & \cdots \arrow[r, "\varphi"] & A_i \arrow[d, "\varphi"] \arrow[r, "\varphi"] & A_{i+1} \arrow[ld, "\id_{A_{i+1}}"] \\
            A_0 \arrow[r, "\varphi"] \arrow[d, "\id_{A_0}"] & A_1 \arrow[r, "\varphi"] \arrow[d, "\id_{A_1}"] & \cdots \arrow[r, "\varphi"] & A_i \arrow[d, "\id_{A_i}"]                    &                             \\
            A_0 \arrow[r, "\varphi"]                                   & A_1 \arrow[r, "\varphi"]                                   & \cdots \arrow[r, "\varphi"]            & A_i \arrow[r, "\varphi"]                                 & A_{i+1}.                    
        \end{tikzcd}
    \end{equation}
    Taking the colimits for the horizontal directions and taking care of those \(A\)-algebra structures, the Frobenius lift \(\varphi\) on the \(\delta\)-ring \(A^{1/p^{i+1}}\) factors through \(t_i\) in the category of \(A\)-algebras as follows:
    \begin{center}
        \begin{tikzcd}
            A^{1/p^{i+1}} \arrow[rr, "\varphi"] \arrow[rd, "\varphi_i"'] \arrow[rd, "\cong"] &    & A^{1/p^{i+1}} \\
             & A^{1/p^i} \arrow[ru, "t_i"'] &              
        \end{tikzcd}
    \end{center}
    We call the isomorphism \(\varphi_i\) of \(A\)-algebras the \emph{\(i\)-th Frobenius lift projection} of the tower of prisms \((\{A^{1/p^i}\}, I)\).

    Taking the quotient in the category of rings (not of \(A\)-algebras), the Frobenius lift \(\varphi\) of \(A^{1/p^{i+1}}\) induces the following commutative diagrams of rings:
    \begin{center}
        \begin{tikzcd}
            {A^{1/p^{i+1}}/(p, I_{i+1})} \arrow[rr, "{\varphi/(p, I)}"] \arrow[rd, "{\varphi_i/(p, I)}"'] \arrow[rd, "\cong"] \arrow[rrr, "F"', bend left] &  & {A^{1/p^{i+1}}/(p, I_{i+1}^{[p]})} \arrow[r, two heads]            & {A^{1/p^{i+1}}/(p, I_{i+1})} \\
            & {A^{1/p^i}/(p, I_i^{[p]})} \arrow[ru, "{t_i/(p, I^{[p]})}"] \arrow[r, two heads] & {A^{1/p^i}/(p, I_i)}. \arrow[ru, "{t_{i, (p, I)}}"'] &                                     
        \end{tikzcd}
    \end{center}
    Here, by the diagram (\ref{FrobeniusProjectionDiag}), note that the left lower isomorphism \(\varphi_i/(p, I)\) is deduced from \(I_i = c^0_i(I)A^{1/p^i}\) and \(\varphi_i(I_{i+1})A^{1/p^i} = \varphi_i(c^0_{i+1}(I)A^{1/p^{i+1}})A^{1/p^i} = c^0_i(\varphi(I))A^{1/p^i} = \varphi(I_i)A^{1/p^i}\) in \(A^{1/p^i}\) by the following commutative diagram:
    \begin{center}
        \begin{tikzcd}
            A \arrow[d, "c_0^{i+1}"'] \arrow[r, "\varphi"]           & A \arrow[d, "c_0^i"] \\
            A^{1/p^{i+1}} \arrow[r, "\varphi_i"] \arrow[r, "\cong"'] & A^{1/p^i}.           
        \end{tikzcd}
    \end{center}
    The lower surjective map
    \begin{equation} \label{FrobeniusProjectionConst}
        \varphi_{i, (p, I)} \colon A^{1/p^{i+1}}/(p, I_{i+1}) \xrightarrow{\varphi_i/(p, I)} A^{1/p^i}/(p, I_i^{[p]}) \twoheadrightarrow A^{1/p^i}/(p, I_i)
    \end{equation}
    becomes a map of \(A\)-algebras, and we call it the \emph{\(i\)-th Frobenius projection} of the tower of prisms \((\{A^{1/p^i}\}, I)\) (or of the tower of rings \((\{A^{1/p^i}/I_i\}, \{t_{i, (p, I)}\})\), see \Cref{PurelyInsepTowerPrism}).
\end{construction}

\begin{observation} \label{ObservationFrobeniusProj}
    Let \((A, I)\) be a preprism.
    Through the isomorphism of towers of \(\delta\)-rings \(\{c_i^i\}\) (\Cref{IsomorphicTower}), we have a commutative diagram in the category of \(\delta\)-rings
    \begin{center}
        \begin{tikzcd}
            A^{1/p^{i+1}} \arrow[r, "\varphi_i"]                                      & A^{1/p^i}                                  \\
            A_{i+1} \arrow[r, "\id_A"] \arrow[u, "c_{i+1}^{i+1}"] \arrow[u, "\cong"'] & A_i. \arrow[u, "\cong"] \arrow[u, "c_i^i"']
        \end{tikzcd}
    \end{center}
    In particular, we can regard the Frobenius lift projection \(\map{\varphi_i}{A^{1/p^{i+1}}}{A^{1/p^i}}\) as the identity map \(\map{\id_A}{A_{i+1}}{A_i}\).
    Furthermore, the isomorphism of towers of rings \(\{c_{i, (p, I)}^i\}\) gives a commutative diagram of rings
    \begin{center}
        \begin{tikzcd}
            {A^{1/p^{i+1}}/(p, I_{i+1})} \arrow[r, "{\varphi_i/(p, I)}"] \arrow[rr, "{\varphi_{i, (p, I)}}", two heads, bend left] \arrow[r, "\cong"'] & {A^{1/p^i}/(p, I_i^{[p]})} \arrow[r, two heads]                      & {A^{1/p^i}/(p, I_i)}                                                                \\
            {A_{i+1}/(p, I^{[p^{i+1}]})A_{i+1}} \arrow[u, "{c_{i+1, (p, I)}^{i+1}}"] \arrow[r, "\id"] \arrow[u, "\cong"'] \arrow[r, "\cong"']          & {A_i/(p, I^{[p^{i+1}]})A_i} \arrow[r, two heads] \arrow[u, "\cong"'] & {{A_i/(p, I^{[p^i]})A_i},} \arrow[u, "{c_{i+1, (p, I)}^{i+1}}"] \arrow[u, "\cong"']
        \end{tikzcd}
    \end{center}
    where the middle vertical isomorphism is induced from \(\map{c_i^i}{A_i}{A^{1/p^i}}\).
    In particular, we can regard the \(i\)-th Frobenius projection \(\varphi_{i, (p, I)} \colon A^{1/p^{i+1}}/(p, I_{i+1}) \twoheadrightarrow A^{1/p^i}/(p, I_i)\) as the natural surjection of rings
    \begin{equation*}
        \pi_i \colon A_{i+1}/(p, I^{[p^{i+1}]})A_{i+1} = A/(p, I^{[p^{i+1}]})A \twoheadrightarrow A/(p, I^{[p^i]})A = A_i/(p, I^{[p^i]})A_i
    \end{equation*}
    defined in \Cref{FrobeniusProjectionQuotient}.
    The Frobenius maps \(F\) on \(A^{1/p^{i+1}}/(p, I_{i+1})\) and \(A_{i+1}/(p, I^{[p^{i+1}]})A_{i+1}\) factor through \(\varphi_{i, (p, I)}\) and \(\pi_i\) as follows:
    \begin{equation} \label{FrobeniusFactorization}
        \begin{tikzcd}
            {A^{1/p^{i+1}}/(p, I_{i+1})} \arrow[rr, "F"] \arrow[rd, "{\varphi_{i, (p, I)}}"', two heads]                                                   &                                                                                                         & {A^{1/p^{i+1}}/(p, I_{i+1})}                                                                    \\
            & {A^{1/p^i}/(p, I_i)} \arrow[ru, "{t_{i, (p, I)}}"']                                                    &                                                                                              \\
            & {A_i/(p, I^{[p^i]})A_i.} \arrow[rd, "\overline{\varphi}_{(p, I)}^{(i)}"] \arrow[u, "{c_{i, (p, I)}^i}"] \arrow[u, "\cong"'] &                                                                                              \\
            {A_{i+1}/(p, I^{[p^{i+1}]})A_{i+1}} \arrow[rr, "F"] \arrow[ru, "\pi_i", two heads] \arrow[uuu, "{c_{i+1, (p, I)}^{i+1}}"] \arrow[uuu, "\cong"'] &                                                                                                         & {A_{i+1}/(p, I^{[p^{i+1}]})A_{i+1}}. \arrow[uuu, "{c_{i+1, (p, I)}^{i+1}}"] \arrow[uuu, "\cong"']
        \end{tikzcd}
    \end{equation}

\end{observation}



At the end of this section, we give the definition of \(p\)-purely inseparable towers and show that the tower of rings generated from prisms is a \(p\)-purely inseparable tower in many cases.

\begin{definition}[{\cite[Definition 3.4]{ishiro2025Perfectoid}}] \label{PurelyInsepTower}
    Let \(R\) be a ring and let \(I_0\) be an ideal of \(R\).
    A tower \((\{R_i\}, \{\iota_i\})\) of rings is called a \emph{\(p\)-purely inseparable tower arising from \((R, I_0)\)} if the following conditions hold.
    \begin{enumalphp}
        \item \(R_0 \cong R\) and \(p \in I_0\).
        \item The induced map \(\map{\overline{\iota_i}}{R_i/I_0R_i}{R_{i+1}/I_0R_{i+1}}\) from \(\iota_i\) is injective for all \(i \geq 0\).
        \item The image of the Frobenius map \(\map{F}{R_{i+1}/I_0R_{i+1}}{R_{i+1}/I_0R_{i+1}}\) is contained in the image of \(\overline{\iota_i}\) for all \(i \geq 0\).
    \end{enumalphp}
    Under these assumptions, the Frobenius map \(\map{F}{R_{i+1}/I_0R_{i+1}}{R_{i+1}/I_0R_{i+1}}\) factors through \(\overline{\iota_i}\) as follows:
    \begin{center}
        \begin{tikzcd}
            R_{i+1}/I_0R_{i+1} \arrow[r, "F"] \arrow[rd, "F_i"'] & R_{i+1}/I_0R_{i+1}               \\
                                                                 & R_i/I_0R_i. \arrow[u, "\overline{\iota_i}"']
        \end{tikzcd}
    \end{center}
    The map \(\map{F_i}{R_{i+1}/I_0R_{i+1}}{R_i/I_0R_i}\) is called the \emph{\(i\)-th Frobenius projection} of the tower \((\{R_i\}, \{\iota_i\})\).
\end{definition}

\begin{lemma} \label{PurelyInsepTowerPrism}
    Let \((A, I)\) be a preprism.
    Assume that the \(p\)-th power map \(\map{\overline{\varphi}_{(p, I)}^{(i)}}{A_i/(p, I^{[p^i]})A_i}{A_{i+1}/(p, I^{[p^{i+1}]})A_{i+1}}\) is injective for all \(i \geq 0\).\footnote{The necessity of this injectivity is a little bit subtle. See \Cref{RemAxiomB}.}
    Then the tower of rings \((\{A_i/\varphi^i(I)A_i\}, \{\overline{\varphi}_I^{(i)}\})\) is a \(p\)-purely inseparable tower arising from \((A/I, (p))\) and its Frobenius projection is nothing but the canonical surjection \(\pi_i \colon A_{i+1}/(p, I^{[p^{i+1}]})A_{i+1} \twoheadrightarrow A_i/(p, I^{[p^i]})A_i\) in \Cref{FrobeniusProjectionQuotient}.

    In particular, the tower of rings \(\{A^{1/p^i}/I_i, \{t_{i, I}\}\}\) is also a \(p\)-purely inseparable tower arising from \((A/I, (p))\) and its Frobenius projection is the Frobenius projection \(\map{\varphi_{i, (p, I)}}{A^{1/p^{i+1}}/(p, I_{i+1})}{A^{1/p^i}/(p, I_i)}\) constructed in \Cref{FrobeniusProjection}.
\end{lemma}

\begin{proof}
    This is clear by the construction of the Frobenius projection in \Cref{FrobeniusProjectionQuotient} (and \Cref{FrobeniusProjection}).
\end{proof}

\section{Construction of Perfectoid Towers from Prisms}
\label{sec:construction_perfectoid_towers}

This section is devoted to our first main result.
We show that the tower of rings \((\{A_i/\varphi^i(I)A_i\}) \cong (\{A^{1/p^i}/IA^{1/p^i}\})\) generated from prism \((A, I)\) becomes a perfectoid tower under mild assumptions and prove its properties (\Cref{PerfectoidTowerPrism}).
To do this, we need some lemmas.

\begin{lemma} \label{Perfectoidpillar}
    Let \((A, I)\) be an orientable preprism and fix an orientation \(d\) of \(I\).
    Set elements
    \begin{align*}
        f_0 \defeq \overline{\varphi(d)}^I \defeq \varphi(d) + I \in A/I, \\
        f_1 \defeq \overline{d}^{\varphi(I)} \defeq d + \varphi(I)A \in A/\varphi(I)A.
    \end{align*}
    If \((A, I)\) is a Zariskian prism (\Cref{DefPrisms}), then the following hold.
    \begin{enumerate}
        \item The ideal \((f_0) \subseteq A/I\) generated by \(f_0\) in \(A/I\) is the same as the ideal \((p) \subseteq A/I\).
        \item There exists a unit element \(u \in (A/\varphi(I)A)^\times\) such that \(f_1^p = u \cdot \overline{\varphi}_I^{(0)}(f_0) \in A/\varphi(I)A\). Here \(\overline{\varphi}_I^{(0)}(f_0) = \varphi(f_0) + \varphi(I)A\) in \(A/\varphi(I)A\) by the definition (\Cref{ConstructionIsom}).
    \end{enumerate}
\end{lemma}

\begin{proof}
    Since \(A\) is \((p, I)\)-Zariskian and \(p\) belongs to \((d, \varphi(d))A\), \(d\) is a distinguished element of \(A\) by \cite[Lemma 2.25]{bhatt2022Prismsa}.

    (1): Passing the equation \(\varphi(d) = d^p + p \delta(d)\) to the quotient \(A/I\), we have \(f_0 = \overline{\varphi(d)}^I = \overline{p \delta(d)}^I\) in \(A/I\). Since \(d\) is a distinguished element, we are done.

    (2): We consider the following equations
    \begin{align*}
        f_1^p & = \overline{d^p}^{\varphi(I)} = \overline{\varphi(d) - p\delta(d)}^{\varphi(I)} = p \cdot \overline{-\delta(d)}^{\varphi(I)} \in A/\varphi(I)A \\
        \overline{\varphi}_I^{(0)}(f_0) & = \overline{\varphi}_I^{(0)}(\overline{\varphi(d)}^{I}) = \overline{\varphi}_I^{(0)}(\overline{p \delta(d)}^I) = p \cdot \overline{\varphi(\delta(d))}^{\varphi(I)} \in A/\varphi(I)A.
    \end{align*}
    Since \(\delta(d)\) is invertible in \(A\), we can take \(u \defeq \overline{-\delta(d) \varphi(\delta(d))^{-1}}^{\varphi(I)} \in (A/\varphi(I)A)^\times\) and we are done.
\end{proof}

\begin{lemma} \label{KernelFrobeniusProjection}
    Let \((A, I)\) be a preprism.
    For the Frobenius projection \(\map{\pi_i}{A_{i+1}/(p, I^{[p^{i+1}]})A_{i+1}}{A_i/(p, I^{[p^i]})A_i}\) defined in (\ref{FrobeniusProjectionQuotient}), we have the following.
    \begin{enumerate}
        \item The kernel of \(\pi_i\) is \(\ker(\pi_i) = I^{[p^{i}]}A_{i+1}/(p, I^{[p^{i+1}]})A_{i+1}\). Via the isomorphism \(c_{i+1}^{i+1}\) in \Cref{IsomorphicColimits}, we also have \(\ker(\varphi_{i, (p, I)}) = c_{1, (p, I)}^{i+1}(I)A^{1/p^{i+1}}/(p, I)A^{1/p^{i+1}}\).
        \item The induced isomorphism \((A^{1/p^{i+1}}/(p, I)A^{1/p^{i+1}})/\ker(\varphi_{i, (p, I)}) \cong A^{1/p^i}/(p, I)A^{1/p^i}\) from \(\varphi_{i, (p, I)}\) is equivalent to the identity map \(A_{i+1}/(p, I^{[p^{i+1}]}, I^{[p^i]})A_{i+1} \xrightarrow{\id} A_i/(p, I^{[p^i]})A_i\) via the isomorphism \(\{c_i^i\}\) in \Cref{IsomorphicTower}.
        \item If \((A, I)\) is orientable, then the kernel \(\ker(\pi_i) \subseteq A_{i+1}/(p, I^{[p^{i+1}]})A_{i+1}\) is generated by \(\overline{f_1}^{p^i} \in A_{i+1}/(p, I^{[p^{i+1}]})\), the image of \(f_1 \in A_1/\varphi(I)A_1\) via the composition \(A_1/\varphi(I)A_1 \twoheadrightarrow A_1/(p, I^{[p]})A_1 \xrightarrow{\overline{\varphi}_{(p, I)}^{(i)} \circ \cdots \circ \overline{\varphi}_{(p, I)}^{(1)}} A_{i+1}/(p, I^{[p^{i+1}]})A_{i+1}\).
        \item If \((A, I)\) is orientable and the \(p\)-th power map \(\map{\overline{\varphi}_{(p, I)}^{(i)}}{A_i/(p, I^{[p^i]})A_i}{A_{i+1}/(p, I^{[p^{i+1}]})A_{i+1}}\) is injective, then the kernel of the Frobenius map \(F\) on \(A_{i+1}/(p, I^{[p^{i+1}]})A_{i+1}\) is generated by \(\overline{f_1}^{p^i}\) as in (3).
    \end{enumerate}
\end{lemma}

\begin{proof}
    (1): The kernel of \(\pi_i \colon A_{i+1}/(p, I^{[p^{i+1}]})A_{i+1} \twoheadrightarrow A_i/(p, I^{[p^i]})A_i\) is nothing but the ideal generated by \(I^{[p^i]}\) in \(A_{i+1}/(p, I^{[p^{i+1}]})A_{i+1}\).
    We have the commutative diagram
    \begin{equation} \label{TowerFactorization}
        \begin{tikzcd}
            {A^{1/p}/(p, I)A^{1/p}} \arrow[rrr, "{t_{i, (p, I)} \circ \cdots \circ t_{1, (p, I)}}"]                                                                                                                                 &  &  & {A^{1/p^{i+1}}/(p, I)A^{1/p^{i+1}}} \arrow[r, "{\varphi_{i, (p, I)}}", two heads]                                          & A^{1/p^i}/(p. I)A^{1/p^i}                                                  \\
            {A_1/(p, I^{[p]})A_1} \arrow[rrru, "{c_{1, (p, I)}^{i+1}}"] \arrow[rrr, "{\overline{\varphi}_{(p, I)}^{(i)} \circ \cdots \circ \overline{\varphi}_{(p, I)}^{(1)}}"'] \arrow[u, "{c_{1, (p, I)}^1}"] \arrow[u, "\cong"'] &  &  & {A_{i+1}/(p, I^{[p^{i+1}]})A_{i+1}} \arrow[u, "{c_{i+1, (p, I)}^{i+1}}"'] \arrow[u, "\cong"] \arrow[r, "\pi_i", two heads] & {A_i/(p, I^{[p^i]})A_i}. \arrow[u, "{c_{i, (p, I)}^i}"'] \arrow[u, "\cong"]
        \end{tikzcd}
    \end{equation}
    The image of \(IA_1/(p, I^{[p]})A\) under the lower horizontal map generates \(\ker(\pi_i) = I^{[p^i]} A_{i+1}/(p, I^{[p^{i+1}]})A_{i+1}\).
    So the kernel of the Frobenius projection \(\varphi_{i, (p, I)} \colon A^{1/p^{i+1}}/(p, I)A^{1/p^i} \twoheadrightarrow A^{1/p^i}/(p, I)A^{1/p^i}\) is
    \begin{equation*}
        c_{i+1, (p, I)}^{i+1}(I^{[p^i]}A_{i+1}/(p, I^{[p^{i+1}]})A_{i+1}) = c_{1, (p, I)}^{i+1}(I)A^{1/p^{i+1}}/(p, I)A^{1/p^{i+1}}.
    \end{equation*}

    (2): This is clear by the commutative diagram (\ref{FrobeniusFactorization}).

    (3): By (1) and our assumption, the kernel \(\ker(\pi_i)\) is generated by \(\overline{d}^{p^i} \in A_{i+1}/(p, I^{[p^{i+1}]})A_{i+1}\) where \(d\) is an orientation of \(I\).

    (4): By \Cref{FrobeniusProjectionFactor}, if \(\overline{\varphi}_{(p, I)}^{(i)}\) is injective, the kernel of the Frobenius map \(F\) on \(A_{i+1}/(p, I^{[p^{i+1}]})A_{i+1}\) is the same as \(\ker(\pi_i)\) and we are done by (3).
\end{proof}

Under these lemmas, we can show that the tower of rings generated from prisms becomes a perfectoid tower. Before that, we recall the definition of perfectoid towers and its related concepts.

\begin{definition}[{Perfectoid towers; \cite[Definition 3.21]{ishiro2025Perfectoid}}] \label{PerfectoidTower}
    Let \(R\) be a ring and let \(I_0\) be an ideal of \(R\).
    A tower \((\{R_i\}, \{\iota_i\})\) of rings is called a \emph{perfectoid tower arising from \((R, I_0)\)} if it is \(p\)-purely inseparable tower arising from \((R, I_0)\) (\Cref{PurelyInsepTower}) and satisfies the following conditions.
    \begin{enumerate}
        \item[(d)] The \(i\)-th Frobenius projection \(\map{F_i}{R_{i+1}/I_0R_{i+1}}{R_i/I_0R_i}\) is surjective for all \(i \geq 0\).
        \item[(e)] Each \(R_i\) is \(I_0\)-Zariskian.
        \item[(f)] The ideal \(I_0\) is a principal ideal in \(R\) and there exists a principal ideal \(I_1\) of \(R_1\) such that \(I_1^{[p]} = I_0R_1\) and the kernel \(\ker(F_i)\) of the \(i\)-th Frobenius projection is generated by the image of \(I_1\) via \(R_1 \twoheadrightarrow R_1/I_0R_1 \to R_{i+1}/I_0R_{i+1}\) for all \(i \geq 0\).
        \item[(g)] Any \(I_0\)-power-torsion element of \(R_i\) is \(I_0\)-torsion, that is, \(R_i[I_0^{\infty}] = R_i[I_0]\). Furthermore, there exists a bijective map \(\map{F_{i, \mathrm{tor}}}{R_{i+1}[I_0^{\infty}]}{R_i[I_0^{\infty}]}\) of sets such that the following diagram commutes:
        \begin{equation} \label{PerfectoidTowerDiagram}
            \begin{tikzcd}
                {R_{i+1}[I_0^{\infty}]} \arrow[r] \arrow[d, "{F_{i, \mathrm{tor}}}"'] & R_{i+1}/I_0R_{i+1} \arrow[d, "F_i"] \\
                {R_i[I_0^{\infty}]} \arrow[r]                                         & R_i/I_0R_i.                         
            \end{tikzcd}
        \end{equation}
    \end{enumerate}
    Such a principal ideal \(I_1\) is uniquely determined and is called the \emph{first perfectoid pillar} of this perfectoid tower (\cite[Definition 3.25]{ishiro2025Perfectoid}).
    Furthermore, we can take a sequence of principal ideal \(\{I_i \subseteq R_i\}_{i \geq 2}\) which satisfies \(F_i(I_{i+1} \cdot R_{i+1}/I_0R_{i+1}) = I_i \cdot R_i/I_0R_i\) for each \(i \geq 0\).
    Such a sequence of principal ideals \(\{I_i\}_{i \geq 2}\) is uniquely determined and \(I_i\) is called the \emph{\(i\)-th perfectoid pillar} of this perfectoid tower \cite[Definition 3.27]{ishiro2025Perfectoid}.
\end{definition}

\begin{definition}[{Tilts of Perfectoid towers; \cite[Definition 3.34]{ishiro2025Perfectoid}}] \label{TiltPerfectoidTower}
    Let \((\{R_i\}, \{\iota_i\})\) be a perfectoid tower arising from \((R, I_0)\).
    The \emph{\(i\)-th small tilt} \((R_i)^{s.\flat}_{I_0}\) (or simply \(R_i^{s.\flat}\)) of \((\{R_i\}, \{\iota_i\})\) associated to \((R, I_0)\) is the inverse limit
    \begin{equation*}
        R_i^{s.\flat} \defeq \lim_{k \geq 0} \{\cdots \xrightarrow{F_{i + k + 1}} R_{i+k+1}/I_0R_{i+k+1} \xrightarrow{F_{i+k}} R_{i+k}/I_0R_{i+k} \xrightarrow{F_{i+k-1}} \cdots \xrightarrow{F_i} R_i/I_0R_i\}
    \end{equation*}
    for each \(i \geq 0\).
    The transition map \(\map{\iota_i^{s.\flat}}{R_i^{s.\flat}}{R_{i+1}^{s.\flat}}\) is the inverse limit of the maps \(\map{\overline{\iota_{i+k}}}{R_{i+k}/I_0R_{i+k}}{R_{i+k+1}/I_0R_{i+k+1}}\) for \(k \geq 0\).
    The \emph{tilt} of the perfectoid tower \((\{R_i\}, \{\iota_i\})\) is the tower \((\{R_i^{s.\flat}\}, \{\iota_i^{s.\flat}\})\).
    The \emph{\(i\)-th small tilt} \(I_i^{s.\flat}\) of the perfectoid pillar \(I_i\) is the kernel
    \begin{equation*}
        I_i^{s.\flat} \defeq \ker(R_i^{s.\flat} \xrightarrow{\Phi_0^{(i)}} R_i/I_0R_i \twoheadrightarrow R_i/I_iR_i)
    \end{equation*}
    for each \(i \geq 0\), where \(\map{\Phi_0^{(i)}}{R_i^{s.\flat}}{R_i/I_0R_i}\) is the first projection.
    By \cite[Proposition 3.41]{ishiro2025Perfectoid}, the tilt \((\{R_i^{s.\flat}\}, \{\iota_i^{s.\flat}\})\) becomes a perfect(oid) tower arising from \((R_0^{s.\flat}, (I_0^{s.\flat}))\).
\end{definition}




Our goal is to construct a perfectoid tower from a class of prisms as follows.

\begin{theorem} \label{PerfectoidTowerPrism}
    Let \((A, I)\) be an orientable Zariskian prism with an orientation \(d \in I\).
    Assume that \(p, d\) is a regular sequence on \(A\),\footnote{Because of the lack of the derived \((p, I)\)-completeness of \(A\), we do not know whether the regular sequence \(p, d\) on \(A\) is permutable or not (see the proof of \cite[Lemma 2.9]{ishizuka2026Prismatic}). However, the two conditions in the statement do not depend on the choice of the orientation \(d\).} and \(A/pA\) is \(p\)-root closed in \(A/pA[1/d]\).
    Then the following assertions hold.
    \begin{enumarabicp}
        \item \((\{R_i\}, \{\iota_i\}) \defeq (\{A_i/\varphi^i(I)A_i\}, \{\overline{\varphi}^{(i)}_I\}) \cong (\{A^{1/p^i}/IA^{1/p^i}\}, \{t_{i, I}\})\) is a perfectoid tower arising from \((A/I, (p))\) whose terms \(A_i/\varphi^i(I)A_i\) are \(p\)-torsion-free. If \(A\) is derived \(p\)-complete or is Noetherian, then each transition map \(\overline{\varphi}^{(i)}_I\) and \(t_{i, I}\) is injective.
        \item Its tilt \((\{(A_i/\varphi^i(I)A_i)^{s.\flat}\}, \{(\overline{\varphi}_{I}^{(i)})^{s.\flat}\})\) is isomorphic to the perfect tower \((\{(A/pA)^{\wedge_d}\}, \{F\})\), where \((-)^{\wedge_d}\) is the \(d\)-adic completion.
        \item The \(p\)-adic completion \(\widehat{R_{\infty}}\) of the colimit \(R_{\infty} \defeq \colim_i R_i\) is isomorphic to the quotient of the perfection of the prism \((A, I)\), that is, isomorphic to the perfectoid ring \(A_{\infty}/IA_{\infty} \cong (A_{\perf}/IA_{\perf})^{\wedge_p}\), where \(A_{\infty}\) is the \((p, I)\)-adic completion of the colimit \(A_{\perf} \defeq \colim_{\varphi} A\).
        \item The first perfectoid pillar \(I_1\) of the tower is \(f_1 A_1/\varphi(I)A_1\), where \(f_1 = \overline{d}^{\varphi(I)} \in A_1/\varphi(I)A_1\) as in \Cref{Perfectoidpillar}.
        \item The \(i\)-th perfectoid pillar \(I_i\) of the tower is \(f_i A_i/\varphi^i(I)A_i\), where \(f_i \defeq \overline{d}^{\varphi^i(I)} = d + \varphi^i(I)A_i \in A_i/\varphi^i(I)A_i\) for each \(i \geq 2\).
        \item The small tilt \(I_i^{s.\flat} \subseteq (A_i/\varphi^i(I)A_i)^{s.\flat}\) of \(I_i\) is isomorphic to \((d) \subseteq A/pA\) for each \(i \geq 0\).
    \end{enumarabicp}
    Note that these statements also hold for any crystalline Zariskian prism \((A, (p))\).
\end{theorem}

\begin{proof}
    In (1), if \(A\) is derived \(p\)-complete or Noetherian, then \(A/\varphi^i(I)A\) is \(p\)-adically separated for any \(i \geq 0\) under our assumption and then the injectivity of \(\overline{\varphi}^{(i)}_I\) follows from \Cref{InjectivityTower}.
    We check the axiom from (a) to (g) of \(p\)-purely inseparable towers (\Cref{PurelyInsepTower}) and perfectoid towers (\Cref{PerfectoidTower}).
    
    (a): This is clear.

    (b): We must show that the \(p\)-th power map \(\overline{\varphi}^{(i)}_{(p, I)}\)
    \begin{equation*}
        A_i/(p, I^{[p^i]})A_i \cong (A/pA)/d^{p^i}(A/pA) \xrightarrow{a \mapsto a^p} (A/pA)/d^{p^{i+1}}(A/pA) \cong A_{i+1}/(p, I^{[p^{i+1}]})A_{i+1}
    \end{equation*}
    is injective for all \(i \geq 0\).
    This condition is equivalent to the condition that \(A/pA\) is \(p\)-root closed in \(A/pA[1/d]\) which we assume now.
    Here we use the assumption that \(d\) is a regular element of \(A/pA\) because we need the injection \(A/pA \hookrightarrow A/pA[1/d]\).

    (c): This is already proved in \Cref{PurelyInsepTowerPrism}.

    (d): The surjectivity of the Frobenius projection \(\map{\pi_i}{A_{i+1}/(p, I^{[p^{i+1}]})A_{i+1}}{A_i/(p, I^{[p^i]})A_i}\) is clear.

    (e): This follows from the \((p, I^{[p^i]})\)-Zariskian property of \(A_i\).

    (f): The principality of \((p)\) in \(A/I\) is clear. By \Cref{Perfectoidpillar} (1), the principal ideal \((p)\) is the same as the principal ideal \((f_0)\) in \(A/I\).
    We set \(I_1 \defeq f_1 A/\varphi(I)A \subseteq A_1/\varphi(I)A_1\) where \(f_1 = \overline{d}^{\varphi(I)} \in A_1/\varphi(I)A_1\) is defined in \Cref{Perfectoidpillar}.
    By \Cref{Perfectoidpillar} (2), we have \(I_1^p = pA_1/\varphi(I)A_1\).
    By \Cref{KernelFrobeniusProjection} (3), the kernel \(\ker(\pi_i)\) of the Frobenius projection \(\pi_i\) is generated by the image of \(I_1\) via \(A_1/\varphi(I)A_1 \twoheadrightarrow A_1/(p, I^{[p]})A_1 \xrightarrow{\overline{\varphi}_{(p, I)}^{(i)} \circ \cdots \circ \overline{\varphi}_{(p, I)}^{(1)}} A_i/(p, I^{[p^{i}]})A_i\).

    (g): Since \(p, d\) is regular on \(A\), so is \(p, \varphi^i(d)\) for each \(i \geq 0\).
    This shows that \(p\) is a non-zero-divisor of \(A_i/\varphi^i(d)A_i\) by a simple calculation.
    Thus, \((A_i/\varphi^i(d)A_i)[p^\infty]\) is zero and the condition (g) is clear.
    This shows (1).

    We compute the tilt of the perfectoid tower \((\{A_i/\varphi^i(I)A_i\}, \{\overline{\varphi}_{I}^{(i)}\})\) arising from \((A/I, (p))\).
    By \Cref{PurelyInsepTowerPrism}, the \(i\)-th small tilt is
    \begin{align*}
        (A_i/\varphi^i(I)A_i)^{s.\flat} & \cong \lim \{\cdots \xrightarrow{\pi_{i+2}} A_{i+2}/(p, I^{[p^{i+2}]})A_{i+2} \xrightarrow{\pi_{i+1}} A_{i+1}/(p, I^{[p^{i+1}]})A_{i+1} \xrightarrow{\pi_i} A_i/(p, I^{[p^i]})A_i\} \\
        & \cong (A/pA)^{\wedge_d},
    \end{align*}
    where the symbol \((-)^{\wedge_d}\) is the \(d\)-adic completion.
    The transition map \(\map{(t_{i, I})^{s.\flat}}{(A/pA)^{\wedge_d}}{(A/pA)^{\wedge_d}}\) is induced from the inverse limit of \(p\)-th power maps \(\overline{\varphi}^{(i)}_{(p, I)}\).
    Then the transition map \((t_{i, I})^{s.\flat}\) is nothing but the Frobenius map \(F\) on \((A/pA)^{\wedge_d}\) and we show (2).

    The colimit \(R_{\infty} \defeq \colim_i R_i\) is isomorphic to \(\colim_{\varphi} (A_i/\varphi^i(I)A_i) \cong A_{\perf}/IA_{\perf}\), where \(A_{\perf} \defeq \colim_{\varphi} A\) as in \cite[Lemma 3.9]{bhatt2022Prismsa}.
    The desired consequence (3) is obtained because \((A_{\perf}/IA_{\perf})^{\wedge_p}\) and \(A_{\infty}/IA_{\infty}\) are isomorphic.

    By the above proof of (1), the first perfectoid pillar \(I_1\) is \(f_1 A_1/\varphi(I)A_1\) where \(f_1 = \overline{d}^{\varphi(I)} \in A_1/\varphi(I)A_1\) as in \Cref{Perfectoidpillar}.
    The \(i\)-th perfectoid pillar \(I_i \subset A_i/\varphi^i(I)A_i\) is \(f_i A_i/\varphi^i(I)A_i\) where \(f_i = \overline{d}^{\varphi^i(I)} = d + \varphi^i(I)A_i \in A_i/\varphi^i(I)A_i\) for each \(i \geq 2\) by the definition of the Frobenius projection \(\pi_i\). This shows (4) and (5).

    The small tilt \(I_i^{s.\flat} \subseteq (A_i/\varphi^i(I)A_i)^{s.\flat}\) of the \(i\)-th perfectoid pillar \(I_i\) is the kernel of the horizontal map of
    \begin{center}
        \begin{tikzcd}
            (A_i/\varphi^i(I)A_i)^{s.\flat} \arrow[d, "\cong"'] \arrow[r, "\Phi_0^{(i)}", two heads] & {A_i/(p, I^{[p^i]})A_i} \arrow[d, Rightarrow, no head] \arrow[r, two heads] & {A_i/(I_i, I^{[p^i]})A_i} \arrow[d, Rightarrow, no head] \\
            (A/pA)^{\wedge_d} \arrow[r, two heads]                                                   & {A/(p, d^{[p^i]})A} \arrow[r, two heads]                                    & {{A/(p, d)A},}                                          
        \end{tikzcd}
    \end{center}
    where the left two horizontal maps are the first projection and the right two horizontal maps are the canonical surjection.
    This shows (6): \(I_i^{s.\flat} \cong (d) \subseteq A/pA\) for any \(i \geq 0\).
\end{proof}

\begin{remark} \label{RemAxiomB}
    In the definition of \(p\)-purely inseparable towers (and of course perfectoid towers), the injectivity axiom (b) in \Cref{PurelyInsepTower} is used to ensure that the tower has the Frobenius projection.
    However, in our case, the tower of rings \((\{A_i/\varphi^i(I)A_i\}, \{\overline{\varphi}_I^{(i)}\})\) has the Frobenius projection \(\pi_i\) without the assumption of the injectivity of the \(p\)-th power map (\Cref{FrobeniusProjectionFactor}).
    Furthermore, all the axioms of perfectoid towers without the axiom (b) are satisfied by the tower \((\{A_i/\varphi^i(I)A_i\}, \{\overline{\varphi}_I^{(i)}\})\) for \emph{any} orientable Zariskian prism \((A, I)\) with an orientation \(d \in I\) such that \(p, d\) is a regular sequence on \(A\).
    By the proof of \Cref{PerfectoidTowerPrism}, the tower \((\{A_i/\varphi^i(I)A_i\}, \{\overline{\varphi}_I^{(i)}\})\) is a perfectoid tower if and only if \(A/pA\) is \(p\)-root closed in \(A/pA[1/d]\).
    Although the axiom (b) is necessary to give a good theory of perfectoid towers, if the existence of Frobenius projection is merely necessary, our construction covers it.
\end{remark}

Based on the above theorem, we come up with the following question which is not clear yet.

\begin{question}
    For a given perfectoid tower \((\{R_i\}, \{t_i\})\), can we construct a prism \((A, I)\) such that the perfectoid tower \((\{A_i/\varphi^i(I)A_i\}, \{\overline{\varphi}^{(i)}_I\})\) is isomorphic to \((\{R_i\}, \{t_i\})\)?
    More optimistically, can we construct a one-to-one correspondence between the set of (\(p\)-torsion-free) perfectoid towers and the set of (specific) prisms by using the above construction?
\end{question}



\section{Perfectoid Towers from \texorpdfstring{\(\delta\)}{delta}-rings}
\label{sec:perfectoid_towers_from_delta_rings}

In this section, for any \(\delta\)-ring \(R\), we make two perfectoid towers \(\{R_i\}_{i \geq 0}\) satisfying \(R_0 \cong R\) or \(R_0 \cong R[\zeta_p]\), where \(\zeta_p\) is a \(p\)-th root of unity.
The one is obtained by adjoining \(p\)-power roots of \(p\) to \(R\) and the other is obtained by adjoining \(p\)-power roots of unity.
Recall that \(R^{1/p^i} \defeq \colim\{R \xrightarrow{\varphi_R} R \xrightarrow{\varphi_R} \cdots \xrightarrow{\varphi_R} R\}\) is the colimit consisting of \((i+1)\)-terms where \(\varphi_R \colon R \to R\) is the Frobenius lift of \(R\).
First we prove a more general result (\Cref{BaseChangePerfectoidTower}) which says that the base change of the tower \(R \to R^{1/p} \to R^{1/p^2} \to \cdots\) along a perfectoid tower induced from a prism is a perfectoid tower.

\subsection{Base change by perfectoid towers}

\begin{notation} \label{ConstTensorProduct}
    In this subsection, we follow the notation below.
    \begin{itemize}
        \item Let \(V\) be an absolute unramified discrete valuation \(\delta\)-ring of mixed characteristic \((0, p)\).
        \item Let \((A, I)\) be an orientable bounded Zariskian prism with an orientation \(d \in I\).
        \item Let \(R\) be a \(p\)-torsion-free \(p\)-Zariskian \(\delta\)-ring.
        \item Write the Frobenius lift on \(A\) by \(\varphi\) and the Frobenius lift on \(R\) by \(\varphi_R\).
        \item Assume that both \(A\) and \(R\) are \(p\)-torsion-free \(\delta\)-\(V\)-algebra.
    \end{itemize}

    In this case, We can identify \(R \otimes_V I\) and \(I (R \otimes_V A)\) as ideals of \(R \otimes_V A\) and the tensor product \(R \otimes_V A\) is flat over \(R\) and \(A\). In particular, it is \(p\)-torsion-free.
    The tensor product \(R \otimes_V A\) and its \((p, I)\)-Zariskization\footnote{More precisely, this is the localization of \(R \otimes_V A\) by \(1 + J\) where \(J \subseteq R \otimes_V A\) is the ideal of \(R \otimes_V A\) generated by \(p\) and \(I\). This localization does not necessarily coincide with the localization of \(R \otimes_V A\) by the image of \(1 + (p, I)A\) in \(A\) through \(A \to R \otimes_V A\).} \((1 + (p, I))^{-1}(R \otimes_V A)\) inherit the unique \(\delta\)-structure compatible with the maps from \(R\) and \(A\). Here we implicitly use \(\varphi((p, I)) = (p, d^p) \subseteq (p, I)\) in \(R \otimes_V A\) and \cite[Lemma 2.15]{bhatt2022Prismsa}.
    We take the flat \(\delta\)-\((R, A)\)-algebra
    \begin{equation} \label{PrismBaseChange}
        P_{A, V}(R) \defeq (1 + (p, I))^{-1}(R \otimes_V A).
    \end{equation}
    The Frobenius lift on \(P_{A, V}(R)\) is denoted by \(\varphi_P\), which is defined as the localization of \(\varphi_R \otimes_{\varphi_V} \varphi\) on \(R \otimes_V A\).
\end{notation}

\begin{lemma} \label{BaseChangeZariskianDeltaRingPrism}
    Keep the notation as in \Cref{ConstTensorProduct}. Then the pair \((P_{A, V}(R), IP_{A, V}(R))\) is an orientable bounded Zariskian prism.
    Moreover, if \(R \otimes_V (A/\varphi^i(I)A)\) is \(p\)-Zariskian for any \(i \geq 0\), there exist natural isomorphisms
    \begin{align*}
        P_{A, V}(R)/\varphi^i(I)P_{A, V}(R) & \cong R \otimes_V (A/\varphi^i(I)A) \\
        P_{A, V}(R)^{1/p^i}/IP_{A, V}(R)^{1/p^i} & \cong R^{1/p^i} \otimes_{V^{1/p^i}} (A^{1/p^i}/IA^{1/p^i})
    \end{align*}
    for each \(i \geq 0\) which are compatible with \(\varphi_P\) and \(\varphi_R \otimes_{\varphi_V} \overline{\varphi}^{(i)}_{I}\). In particular, if \(A/\varphi^i(I)A\) is integral over \(V\), this isomorphism holds.
\end{lemma}

\begin{proof}
    We already know that \(P_{A, V}(R)\) is a \(\delta\)-ring and \(p \in IP_{A, V}(R) + \varphi(I)P_{A, V}(R)\). Since \(P_{A, V}(R)\) is flat over \(A\), the ideal \(IP_{A, V}(R)\) is generated by a non-zero-divisor \(1 \otimes d\) in \(P_{A, V}(R)\).
    The \(p^{\infty}\)-torsion part of \(P_{A, V}(R)\) is the base change of the \(p^{\infty}\)-torsion part of \(A\) because of the flatness of \(R\) over \(V\).
    The construction of \(P_{A, V}(R)\) ensures that it is \((p, I)\)-Zariskian and therefore the pair \((P_{A, V}(R), IP_{A, V}(R))\) is an orientable bounded Zariskian prism.

    If \(R \otimes_V (A/\varphi^i(I)A)\) is \(p\)-Zariskian, then the isomorphism \(P_{A, V}(R)/\varphi^i(I)P_{A, V}(R) \cong (1 + I)^{-1}(R \otimes_V (A/\varphi^i(I)A))\) holds. It suffices to show that the \(p\)-Zariskian ring \(R \otimes_V (A/\varphi^i(I)A)\) is \(I\)-Zariskian but this follows from the equation \((p, \varphi^i(I)) = (p, d^{p^i})\) on \(A\).
\end{proof}

Our second main result is \Cref{BaseChangePerfectoidTower} below. This says that the base change of a perfectoid tower \((\{A_i/\varphi^i(I)A_i\})\) along a tower \(\{R^{1/p^i}\}\) of a \(\delta\)-ring \(R\) is again a perfectoid tower.
To check the \(p\)-root closed property, we need the following lemma.

\begin{lemma} \label{pRootClosedLocalization}
    Let \(A\) be a ring and let \(t \in A\) be a non-zero-divisor.
    Take a multiplicative closed subset \(S \subseteq A\) consisting of some non-zero-divisors.
    If \(A\) is \(p\)-root closed in \(A[1/t]\), then \(S^{-1}A\) is \(p\)-root closed in \(S^{-1}A[1/t]\).
\end{lemma}

\begin{proof}
    By our assumption, we have injections \(A \hookrightarrow A[1/t] \hookrightarrow S^{-1}A[1/t]\).
    Let \(a/st^n \in S^{-1}A[1/t]\) be an element satisfying that \((a/st^n)^p = a'/s' \in S^{-1}A\).
    Then we have \((as'/t^n)^p = s^pa'(s')^{p-1} \in A\) and thus \(as'/t^n \in A\) by the \(p\)-root closedness of \(A\) in \(A[1/t]\).
    This shows that \(a/st^n = 1/ss' \cdot as'/t^n \in S^{-1}A\) and the lemma is proved.
\end{proof}

\begin{theorem} \label{BaseChangePerfectoidTower}
    Keep the notation as in \Cref{ConstTensorProduct} and assume that \(R/pR\) is reduced.
    If further \((A, (d))\) satisfies the assumptions of \Cref{PerfectoidTowerPrism},\footnote{Namely, \(p, d\) is a regular sequence on \(A\) and \(A/pA\) is \(p\)-root closed in \(A/pA[1/d]\).} then so does the orientable Zariskian prism \((P_{A, V}(R), IP_{A, V}(R))\) in \Cref{BaseChangeZariskianDeltaRingPrism}.
    In particular, the following assertions hold.
    \begin{enumerate}
        \item The Zariskian prism \((P_{A, V}(R), IP_{A, V}(R))\) gives a \(p\)-torsion-free perfectoid tower \((\{(1 + (p))^{-1}(R \otimes_V (A/\varphi^i(I)A))\}, \{\varphi_R \otimes_{\varphi_V} \overline{\varphi}_I^{(i)}\}) \cong (\{(1 + (p))^{-1}(R^{1/p^i} \otimes_{V^{1/p^i}} (A^{1/p^i}/IA^{1/p^i}))\}, \{\varphi_R \otimes_{\varphi_V} t_{i, I}\})\) arising from \(((1+(p))^{-1}(R \otimes_V (A/I)), (p))\).
        \item Its tilt is isomorphic to the perfect tower \((\{(R/pR \otimes_{V/pV} A/pA)^{\wedge_d}\}, \{F\})\).
        \item The \(p\)-completed colimit of the perfectoid tower is isomorphic to the \(p\)-completed tensor product \(R_{\perf} \widehat{\otimes}_{V} A_{\perf}/IA_{\perf}\).
        \item The \(i\)-th perfectoid pillar of the tower is generated by \(1 \otimes f_i \in R \otimes_V A/\varphi^i(I)A\) where \(f_i = \overline{d}^{\varphi^i(I)} = d + \varphi^i(I)A \in A/\varphi^i(I)A\) as in \Cref{PerfectoidTowerPrism} for \(i \geq 1\).
        \item The \(i\)-th small tilt of the \(i\)-th perfectoid pillar is isomorphic to \((1 \otimes d) \subseteq (R/pR \otimes_{V/pV} A/pA)^{\wedge_d}\) for \(i \geq 0\).
    \end{enumerate}
    Moreover, if \(R \otimes_V (A/\varphi^i(I)A)\) is \(p\)-Zariskian for all \(i \geq 0\), then we have a perfectoid tower \((\{R \otimes_V (A/\varphi^i(I)A)\}, \{\varphi_R \otimes \overline{\varphi}^{(i)}_I\}) \cong (\{R^{1/p^i} \otimes_V (A^{1/p^i}/IA^{1/p^i})\}, \{\varphi_R \otimes t_{i, I}\})\) arising from \((R \otimes_V A/I, (p))\). 
\end{theorem}

\begin{proof}
    By \Cref{PerfectoidTowerPrism} and \Cref{BaseChangeZariskianDeltaRingPrism}, it is enough to show that \(p, d\) is a regular sequence on \(P_{A, V}(R)\) and \(P_{A, V}(R)/(p) \cong (1 + I)^{-1}(R/pR \otimes_{V/pV} A/pA)\) is \(p\)-root closed in \(P_{A, V}(R)/(p)[1/d]\).
    The first assertion follows from the regularity of \(p, d\) on \(A\) and the flatness of \(R \otimes_V A\) over \(A\).
    In the second assertion, \Cref{pRootClosedLocalization} implies that it is enough to show that \(R/pR \otimes_{V/pV} A/pA\) is \(p\)-root closed in \((R/pR \otimes_{V/pV} A/pA)[1/d]\).
    As in the proof of \Cref{PerfectoidTowerPrism}, this is equivalent to the injectivity of the \(p\)-th power map \((R/pR \otimes_{V/pV} A/pA)/(d) \xrightarrow{x \mapsto x^p} (R/pR \otimes_{V/pV} A/pA)/(d^p)\). This map can be factorized as the composition
    \begin{equation*}
        R/pR \otimes_{V/pV} A/(p, d)A \xrightarrow{\id \otimes \overline{\varphi}^{(0)}_{(p, I)}} R/pR \otimes_{V/pV} A/(p, d^p)A \xrightarrow{\Frob \otimes \id} R/pR \otimes_{V/pV} A/(p, d^p)A.
    \end{equation*}
    The former map is injective by the \(p\)-root closedness of \(A/pA\) in \(A/pA[1/d]\) and the flatness of \(R\) over \(V\).
    Since \(V\) is an absolute unramified discrete valuation ring, \(V/pV\) is a field and thus the latter map is injective by the reducedness of \(R/pR\).

    By \Cref{BaseChangeZariskianDeltaRingPrism}, we have a natural isomorphism \((1 + (p, I))^{-1}(R \otimes_V (A/\varphi^i(I)A)) \cong R \otimes_V (A/\varphi^i(I)A)\).
    So we can deduce the last assertion from the first assertion.
\end{proof}



\subsection{Adjoining \(p\)-power roots of \(p\) and unity}

Based on \Cref{BaseChangePerfectoidTower}, we can construct perfectoid towers by adjoining \(p\)-power roots of \(p\) and unity to \(\delta\)-rings.

\begin{corollary} \label{ExplicitPerfectoidTowerDeltaRing}
    Let \(R\) be a \(p\)-torsion-free \(p\)-Zariskian \(\delta\)-ring such that \(R/pR\) is reduced. Fix compatible sequences \(\{p^{1/p^i}\}_{i \geq 0}\) and \(\{\zeta_{p^i}\}_{i \geq 0}\) of \(p\)-power roots of \(p\) and unity in \(\overline{\setQ}\).
    Then we have the following \(p\)-torsion-free perfectoid towers:
    \begin{align}
        R & \to R^{1/p} \otimes_{\setZ} \setZ[p^{1/p}] \to \cdots \to R^{1/p^i} \otimes_{\setZ} \setZ[p^{1/p^i}] \to \cdots \label{PerfectoidTowerDeltaRing1} \\
        R \otimes_{\setZ} \setZ[\zeta_p] & \to R^{1/p} \otimes_{\setZ} \setZ[\zeta_p^{1/p}] \to \cdots \to R^{1/p^i} \otimes_{\setZ} \setZ[\zeta_p^{1/p^i}] \to \cdots \label{PerfectoidTowerDeltaRing2}     
    \end{align}
    arising from \((R, (p))\) and \((R \otimes_{\setZ} \setZ[\zeta_p], (p))\) respectively. If \(R\) is \(p\)-adically separated, their transition maps are injective.
    Their \(p\)-completed colimits are isomorphic to \((R_{\perf} \otimes_{\setZ} \setZ[p^{1/p^{\infty}}])^{\wedge_p}\) and \((R_{\perf} \otimes_{\setZ_p} \setZ_p[\zeta_{p^\infty}])^{\wedge_p}\), respectively.
    Both tilts are isomorphic to
    \begin{equation*}
        R/pR[|T|] \xrightarrow{F} R/pR[|T|] \xrightarrow{F} R/pR[|T|] \xrightarrow{F} \cdots
    \end{equation*}
    where \(F\) is the Frobenius map on the formal power series ring \(R/pR[|T|]\).

    If further \(R\) is \(p\)-adically separated, \(R\) and \(R/pR\) are integral domains and the Frobenius lift \(\varphi_R\) is finite, then the perfectoid towers (\ref{PerfectoidTowerDeltaRing1}) and (\ref{PerfectoidTowerDeltaRing2}) are isomorphic to the towers of subrings
    \begin{align*}
        R \hookrightarrow R^{1/p}[p^{1/p}] \hookrightarrow R^{1/p^2}[p^{1/p^2}] \hookrightarrow \cdots \hookrightarrow R^{1/p^i}[p^{1/p^i}] \hookrightarrow \cdots \\
        R \hookrightarrow R^{1/p}[\zeta_{p}] \hookrightarrow R^{1/p^2}[\zeta_{p^2}] \hookrightarrow \cdots \hookrightarrow R^{1/p^i}[\zeta_{p^i}] \hookrightarrow \cdots
    \end{align*}
    in a fixed absolute integral closure \(R^+\) of \(R\). Here we take an embedding of a finite extension \(R^{1/p^i} \defeq \colim\{R \xrightarrow{\varphi} R \xrightarrow{\varphi} \cdots \xrightarrow{\varphi} R\}\) of \(R\) into \(R^+\) for each \(i \geq 0\).
\end{corollary}

\begin{proof}
    The orientable bounded Zariskian prisms \(((1 + (T))^{-1}\setZ_{(p)}[T], (p-T))\) and \(((1+(q-1))^{-1}\setZ_{(p)}[q], ([p]_q))\) in \Cref{ExampleZariskianPrism} satisfy the assumption of \Cref{PerfectoidTowerPrism}.
    Applying \Cref{BaseChangePerfectoidTower} for \(V \defeq \setZ_{(p)}\), the isomorphisms \(\setZ[T]/(p-T^{p^i}) \cong \setZ[p^{1/p^i}]\) and \(\setZ[q]/([p]_q) \cong \setZ[\zeta_p]\) yields the perfectoid towers and their tilts are isomorphic to the tower \((\{R/pR[|T|]\}, \{F\})\).
    If \(R\) is \(p\)-adically separated, \Cref{DeltaInjective} implies that the Frobenius lift \(\varphi_R\) is injective.
    Since \(\setZ \hookrightarrow \setZ[p^{1/p^i}]\) and \(\setZ \hookrightarrow \setZ[\zeta_{p^i}]\) are flat, the transition maps in the towers are injective.


    We next prove the last assertion. The isomorphism \(R^{1/p^i}/pR^{1/p^i} \cong (R/pR)^{1/p^i}\) holds for each \(i \geq 0\) and thus \(p\) is a prime element of \(R^{1/p^i}\).
    Since \(\varphi_R\) is finite injective, the canonical map \(R^{1/p^i} \hookrightarrow R^{1/p^{i+1}}\) is also finite injective.
    This gives a sequence of finite extensions of integral domains
    \begin{equation*}
        R \hookrightarrow R^{1/p} \hookrightarrow R^{1/p^2} \hookrightarrow \cdots \hookrightarrow R^{1/p^i} \hookrightarrow \cdots.
    \end{equation*}
    This sequence is contained in a fixed absolute integral closure \(R^+\) of \(R\).
    By (\ref{PerfectoidTowerDeltaRing1}) and (\ref{PerfectoidTowerDeltaRing2}), we need to show that the canonical maps \(R^{1/p^i}[T]/(p-T^{p^i}) \cong R^{1/p^i} \otimes_{\setZ} \setZ[p^{1/p^i}] \twoheadrightarrow R^{1/p^i}[p^{1/p^i}]\) and \(R^{1/p^i}[q^{1/p^i}]/([p]_q) \cong R^{1/p^i} \otimes_{\setZ_p} \setZ_p[\zeta_{p^i}] \twoheadrightarrow R^{1/p^i}[\zeta_{p^i}]\) are isomorphisms.
    It is enough to show that \(T^{1/p^i}-p\) and \([p]_q\) are irreducible polynomials in \(R^{1/p^i}[T]\) and \(R^{1/p^i}[q^{1/p^i}]\), respectively.
    This follows from the fact that \(p\) is a prime element of \(R^{1/p^i}\) and the Eisenstein criterion for \(T^{1/p^i}-p\) and the variable transformation \(((q+1)^p-1)/q\) of \([p]_q\).
\end{proof}

\begin{remark} \label{PerfectoidTowerSingularity}
    In Noetherian case, such \(\delta\)-ring \(R\) in \Cref{ExplicitPerfectoidTowerDeltaRing} relates the Frobenius liftable singularities: Let \(A\) be a reduced Noetherian local ring over a perfect field \(k\) of characteristic \(p\).
    If \(A\) is Frobenius liftable, namely, there exists a flat \(W(k)\)-algebra \(R\) with a Frobenius lift \(\varphi_R\) such that \(R/pR \cong A\), then there exists a perfectoid tower arising from \((\widehat{R}, (p))\) and its tilt is the perfect tower arising from \((A[|T|], (p))\) where \(\widehat{R}\) is the \(p\)-adic completion of \(R\).
    In fact, \(\widehat{R}\) is a \(p\)-torsion-free \(p\)-adically complete Noetherian local \(\delta\)-ring such that \(R/pR \cong A\) is reduced (here we use \citeSta{0G5H}) and we can apply \Cref{ExplicitPerfectoidTowerDeltaRing}.
    We use this observation in \Cref{ExampleDeltaStabilization}.
\end{remark}

\subsection{Generic ranks of transition maps of perfectoid towers} \label{GenericRankPerfectoidTower}

If we know the generic rank of transition maps of a perfectoid tower is \(p\)-power, then some \'etale cohomology of mixed characteristic can be captured by the one of positive characteristic (see \cite[Proposition 4.7]{ishiro2025Perfectoid}).
In general, generic ranks are not easily computable, but we give some sufficient conditions in an algebraic situation and a geometric situation.
One of the most simple case is the following.

\begin{lemma} \label{GenericDegreePrism}
    Let \((A, I)\) be an orientable Zariskian prism with an orientation \(d \in I\).
    If \(\varphi\) on \(A\) is finite free of degree \(\deg \varphi\) and \(A/\varphi^i(I)A\) is an integral domain for each \(i \geq 0\), then \(\deg \varphi\) is the degree of the generic extension of the transition map \(\overline{\varphi}_I^{(i)} \colon A/\varphi^i(I)A \hookrightarrow A/\varphi^{i+1}(I)A\) for each \(i \geq 0\).
\end{lemma}

In the case of \(\delta\)-rings, there is a similar result.

\begin{lemma} \label{GenericDegreeDeltaRing}
    Let \(R\) be a \(p\)-Zariskian \(p\)-adically separated \(\delta\)-ring such that \(R\) and \(R/pR\) are integral domains and the Frobenius lift \(\varphi_R\) is finite.
    Let \(\deg \varphi_R\) be the degree of the generic extension of the Frobenius lift \(\varphi_R\) on \(R\).
    Then the degree of the generic extension of \(R^{1/p^i}[p^{1/p^i}] \hookrightarrow R^{1/p^{i+1}}[p^{1/p^{i+1}}]\) is \(p \cdot \deg \varphi_R\) for any \(i \geq 0\).
\end{lemma}

\begin{proof}
    Set \(K \defeq \Frac(R)\). First we can show that \(K\) and \(\setQ[p^{1/p^i}]\) are linearly independent over \(\setQ\) in an algebraic closure \(\overline{K}\) of \(K\): If elements \(x_0, \dots, x_i\) in \(K\) satisfies \(x_0 + x_1 p^{1/p^i} + \cdots + x_{p^i-1} p^{(p^i-1)/p^i} = 0\) in \(\overline{K}\), then we may assume that \(x_0, \dots, x_i\) are in \(R\) and so \(x_0 + x_1 T + \cdots + x_{p^i-1} T^{p^i-1} = 0\) in \(R[T]\) since we have an isomorphism \(R[T]/(p-T^{p^i}) \cong R[p^{1/p^i}]\) as in the proof of \Cref{ExplicitPerfectoidTowerDeltaRing}. Therefore, \(x_0 = \cdots = x_{p^i-1} = 0\) in \(R\) and so in \(K\). This shows the linear independence.
    Especially this implies \(\Frac(R[p^{1/p^i}]) \cong K[p^{1/p^i}] \cong K \otimes_{\setQ} \setQ[p^{1/p^i}]\).
    Since we have an isomorphism \(R^{1/p^i} \otimes_{\setZ} \setZ[p^{1/p^i}] \cong R^{1/p^i}[p^{1/p^i}]\), it suffices to consider the generic rank of the map \(R \otimes_{\setZ} \setZ[p^{1/p^i}] \to (\varphi_*R) \otimes_{\setZ} \setZ[p^{1/p^{i+1}}]\).
    Since the degree \([K(\varphi_*R): K(R)]\) is \(\deg \varphi_R\), we have
    \begin{equation*}
        \Frac((\varphi_*R) \otimes_{\setZ} \setZ[p^{1/p^{i+1}}]) \cong K(\varphi_*R) \otimes_{\setQ} \setQ[p^{1/p^{i+1}}] \cong K^{\oplus \deg \varphi_R} \otimes_{\setQ} \setQ[p^{1/p^{i+1}}]
    \end{equation*}
    and this is a \((p \cdot \deg \varphi_R)\)-dimensional \((K \otimes_{\setQ} \setQ[p^{1/p^i}])\)-vector space.
\end{proof}

We give a sufficient condition of \(p\)-power generic rank of \(\varphi\) in both algebraic and geometric situation.

\begin{proposition} \label{GenericDegreeDeltaRingNormal}
    Let \(R\) be a \(\delta\)-ring such that \(R\) and \(R/pR\) are integral domains and the Frobenius lift \(\varphi_R\) on \(R\) is finite.
    If \(R\) is Noetherian and normal, then the generic rank of \(\varphi_R\) is \(p\)-power.
\end{proposition}

\begin{proof}
    By \Cref{DeltaInjective}, the Frobenius lift \(\varphi_R\) is finite injective.
    Since \(\mfrakp \defeq pR\) is a prime ideal of \(R\) and \(R\) is Noetherian normal, taking the localization at \(p\), we have a finite injective map \(\varphi_{\mfrakp} \colon R_{\mfrakp} \hookrightarrow R_{\mfrakp}\) between discrete valuation rings.
    In particular, this map \(\varphi_{\mfrakp}\) is a finite flat endomorphism of a discrete valuation ring.
    Therefore, \(\varphi_{\mfrakp, *}R_{\mfrakp}\) is a finite free \(R_{\mfrakp}\)-module.
    Thus the generic rank of \(\varphi_{\mfrakp}\) is the same as the generic rank of the Frobenius map on \(K(R/pR) = R_{\mfrakp}/pR_{\mfrakp}\) which is \(p\)-power.
    Since the generic rank of \(\varphi_R\) is the same as that of \(\varphi_{\mfrakp}\), the generic rank of \(\varphi_R\) is \(p\)-power.
\end{proof}

\begin{proposition} \label{ExampleGraded}
    Let \(R = \oplus_{i \geq 0}R_i\) be a \(p\)-torsion-free graded \(W(k)\)-algebra with an endomorphism \(\varphi\) on \(R\) which induces the Frobenius map on \(R/pR\) and \(\varphi \otimes_{\setZ} \setZ/p^n\setZ\) sends \((R/p^nR)_i\) to \((R/p^nR)_{ip}\) for any \(n \geq 0\) and \(i \geq 0\).
    Assume that \(\mcalX \defeq \Proj(R)\) is a smooth projective \(W(k)\)-scheme.
    Then the generic rank of the Frobenius lift \(\varphi\) on \(R\) is \([K(X)^{1/p}:K(X)]\) where \(X \defeq \mcalX_{p=0}\) and in particular \(p\)-power.
\end{proposition}

\begin{proof}
    Recall that the Frobenius lift \(\varphi_R\) induces an endomorphism of schemes \(\widetilde{F}_{\mcalX} \colon \mcalX \to \mcalX\) and this is compatible with the Frobenius lift on \(\Spec(W(k))\): \(\varphi_R\) induces a morphism of schemes \(U \to \mcalX = \Proj(R)\) from an open subscheme \(U \defeq \set{\mfrakp \in \mcalX}{\varphi_R^{-1}(\mfrakp) \nsupseteq R_+}\) in \(\mcalX\) (\cite[(13.2.4)]{gortz2020Algebraic}). Since \(U\) contains the special fiber \(X \defeq \mcalX \times_{W(k)} k\), \(U\) should be the whole \(\mcalX\) and thus \(\widetilde{F}_{\mcalX}\) is well-defined.

    Taking the special fiber, the Frobenius map on \(R/pR\) induces the Frobenius morphism \(F_X \colon X \to X\) of a smooth \(k\)-scheme \(X\) (through the Frobenius map on \(k\)) and in particular this is a finite morphism.
    This says that the restriction of \(\widetilde{F}_{\mcalX} \colon \mcalX \to \mcalX\) to \(X\) has \(0\)-dimensional fibers and this property extends to \(\mcalX\). Consequently, \(\widetilde{F}_{\mcalX}\) is a finite morphism (\citeSta{02OG}) and thus \(\widetilde{F}_{\mcalX}\) is flat by miracle flatness (\citeSta{00R4}).
    In particular, \(\widetilde{F}_{\mcalX}\) is finite locally free by \cite[Proposition 12.19]{gortz2020Algebraic}.
    By \cite[Proposition 13.37 (2)]{gortz2020Algebraic}, we have \(R_f \cong R_{(f)}[T, T^{-1}]\) for any homogeneous element \(f \in R_1\). Since \(\widetilde{F}_{\mcalX}\) induces a finite free map \(R_{(f)} \to R_{(\varphi(f))}\) between regular rings for some \(f \in R_1\), the generic rank \(\deg \varphi\) is the same as the rank of the finite free map \(R_f \to R_{\varphi(f)}\) induced from \(\varphi\). This is the same as the rank of the Frobenius map \((R/pR)_f \xrightarrow{F} (R/pR)_f\) since \(D_+(f) \cap \mcalX_{p=0} \neq \emptyset\).
    Therefore, the generic degree \(\deg \varphi\) is the \(p\)-power \([K(X)^{1/p}: K(X)]\).
\end{proof}

\section{Examples} \label{sec:examples}

We present some examples of perfectoid towers generated from prisms.
While the known examples do not arise from mild singularities such as (log-)regularity, the first term of the following examples are not necessarily log-regular.
Before giving new examples, we reconstruct the previous examples.

\subsection{Previous examples of perfectoid towers}

These examples were calculated separately in \cite{ishiro2025Perfectoid}, but now can be treated in a unified and simpler fashion using our first main theorem (\Cref{PerfectoidTowerPrism}). The first two examples are generalized in \Cref{ExampleDeltaStabilization} later.
Note that the first two examples are only examples of Noetherian perfectoid towers previously known in \cite{ishiro2025Perfectoid} and these are perfectoid towers arising from Cohen-Macaulay normal domains.

\begin{example}[Regular rings; {\cite[Example 3.62 (1)]{ishiro2025Perfectoid}}] \label{RegularExample}
    Any complete regular local ring \(R\) of dimension \(d\) with residue field \(k\) of characteristic \(p > 0\) can be represented as \(A/I\) for some complete regular local prism \((A, I)\) (see \cite[Corollary 3.8]{ishizuka2026Prismatic}).
    In this case, \(A \cong C(k)[|T_1, \dots, T_d|]\) and \(I = (p-f)\), where \(C \defeq C(k)\) is the Cohen ring of \(k\) equipped with a \(\delta\)-ring structure and \(f \in (T_1, \dots, T_n)\) (see \cite[Lemma 2.6 and Lemma 5.1]{ishizuka2026Prismatic}).
    We denote the colimit \(C^{1/p^i} \defeq \colim \{C \xrightarrow{\varphi} C \xrightarrow{\varphi} \cdots \xrightarrow{\varphi} C\}\) consisting of \(i+1\) terms. It gives extensions of integral domains \(C \hookrightarrow C^{1/p} \hookrightarrow C^{1/p^2} \hookrightarrow \).
    Applying \Cref{PerfectoidTowerPrism}, we obtain a perfectoid tower
    \begin{equation*}
        R \cong C[|T_1, \dots, T_d|]/(p-f) \xrightarrow{t_{1, I}} \cdots \xrightarrow{t_{i, I}} C^{1/p^i}[|T_1^{1/p^i}, \dots, T_d^{1/p^i}|]/(p-f) \xrightarrow{t_{i+1, I}} \cdots
    \end{equation*}
    whose transition maps are injective and its tilt is isomorphic to the perfect tower
    \begin{equation*}
        k[T_1, \dots, T_d] \hookrightarrow k^{1/p}[|T_1^{1/p}, \dots, T_d^{1/p}|] \hookrightarrow \cdots \hookrightarrow k^{1/p^i}[|T_1^{1/p^i}, \dots, T_d^{1/p^i}|] \hookrightarrow \cdots.
    \end{equation*}
    If \(f = T_d\), the quotient \(R \cong A/I\) is \(C[|T_1, \dots, T_{d-1}|]\) and the above perfectoid tower is the same as the perfectoid tower generated from the \(\delta\)-ring \(C[|T_1, \dots, T_{d-1}|]\) (\Cref{ExplicitPerfectoidTowerDeltaRing}).
\end{example}

\begin{example}[Local log-regular rings; {\cite[\S 3.6]{ishiro2025Perfectoid}}] \label{LogRegularExample}
    More generally, our construction covers the case of complete local log-regular rings:
    let \(C\) be the Cohen ring of a field \(k\) of positive characteristic \(p\) and let \(\mcalQ\) be a fine sharp saturated monoid.
    Fix a \(\delta\)-ring structure of \(C\) as above.
    Then the complete Noetherian local domain \(C[|\mcalQ|]\) has a \(\delta\)-structure given by the Frobenius lift \(e^q \mapsto (e^q)^p\) and \((C[|\mcalQ|], (p-f))\) becomes an orientable prism for any \(f \in C[|\mcalQ|]\) which has no non-zero constant term (or \(f = 0\)) as above.
    By Kato's structure theorem, any complete local log-regular ring \((R, \mcalQ, \alpha)\) of residue characteristic \(p\) can be represented as \(C[|\mcalQ|]/(p-f)\) for some \(f \in C[|\mcalQ|]\) which has no non-zero constant term (see, for example, \cite[Theorem 2.22]{ishiro2025Perfectoid}).
    
    Since \(C[|\mcalQ|]/(p-f)\) is a complete local log-regular ring, \((C[|\mcalQ|], (p-f))\) is transversal and \(k[|\mcalQ|]\) is \(p\)-root closed in \(k[|\mcalQ|][1/\overline{f}]\).
    Then by \Cref{PerfectoidTowerPrism}, the tower
    \begin{equation*}
        C[|\mcalQ|]/(p-f) \hookrightarrow C^{1/p}[|\mcalQ^{(1)}|]/(p-f) \hookrightarrow \cdots \hookrightarrow C^{1/p^i}[|\mcalQ^{(i)}|]/(p-f) \hookrightarrow \cdots
    \end{equation*}
    is a perfectoid tower arising from \((C[|\mcalQ|]/(p-f), (p))\) and those transition maps are injective.
    Its tilt is
    \begin{equation*}
        k[|\mcalQ|] \hookrightarrow k[|\mcalQ^{(1)}|] \hookrightarrow \cdots
    \end{equation*}
    since the Frobenius map on \(k[|\mcalQ|]\) is compatible with the inclusion \(k[|\mcalQ|] \hookrightarrow k[|\mcalQ^{(1)}|]\).
    The resulting perfectoid tower and its tilt are the same as those in \cite[Proposition 3.58, Lemma 3.59, Theorem 3.61, and Example 3.62]{ishiro2025Perfectoid}.

\end{example}

The above examples appear in commutative ring theory and the following examples are related to arithmetic geometry, in particular, prismatic theory.

\begin{example}[Perfectoid rings; {\cite[Theorem 3.10]{bhatt2022Prismsa}}]
    Let \(R\) be a \(p\)-torsion-free perfectoid ring.
    Then there exists a unique transversal perfect prism \((A, (\xi))\) such that \(R \cong A/(\xi)\).
    The assumption of \Cref{PerfectoidTowerPrism} is satisfied because of the \(p\)-torsion-free property of \(A\) and perfectness of \(A/pA\).
    Since the canonical map \(\map{c_0^i}{A}{A^{1/p^i}}\) is isomorphism, we have \(A^{1/p^i}/IA^{1/p^i} \cong A/I\). By \Cref{PerfectoidTowerPrism}, the following tower is a perfectoid tower
    \begin{equation*}
        R \cong A/I \xrightarrow{\id} A/I \xrightarrow{\id} \cdots \xrightarrow{\id} A/I \xrightarrow{\id} \cdots
    \end{equation*}
    and its tilt is
    \begin{equation*}
        R^\flat \cong A/pA \xrightarrow{F} A/pA \xrightarrow{F} \cdots \xrightarrow{F} A/pA \xrightarrow{F} \cdots.
    \end{equation*}
    This is the case of the perfectoid tower from (\(p\)-torsion-free) perfectoid rings \cite[Example 3.53]{ishiro2025Perfectoid}.
\end{example}


\subsection{Examples from geometric Frobenius lifts}

The next example is generated from a more geometric methods, namely, the Frobenius lift on an Abelian variety. That tower is an example that the generic rank of those transition maps are \(p\)-power.
More general theory of Frobenius lifts on smooth projective varieties and its relation to perfectoid towers are developed in \cite{ishizuka2025Quasi, ishiro2025drings}.
This is one of the examples of perfectoid towers arising from non-Cohen-Macaulay normal domains, which does not appear in \cite{ishiro2025Perfectoid}.

\begin{example} \label{ExampleCone}
    Let \(A\) be an ordinary Abelian variety over a perfect field \(k\) of characteristic \(p\) and \(L\) be an ample line bundle on \(A\).
    Then we can take the canonical lift \(\mcalA\) and an ample line bundle \(\mcalL\) on \(\mcalA\) such that the ring of sections \(R(\mcalA, \mcalL) \defeq \oplus_{m \geq 0} H^0(\mcalA, \mcalL^{\otimes m})\) is a normal domain but not Cohen-Macaulay (as in \cite[Lemma 4.11]{kawakami2024Frobenius} and \cite[Example 7.7]{bhatt2024Perfectoid}).
    Then this ring \(R(\mcalA, \mcalL)\) satisfies the condition \Cref{ExplicitPerfectoidTowerDeltaRing} and \Cref{ExampleGraded}. So we have a perfectoid tower
    \begin{equation*}
        R(\mcalA, \mcalL) \hookrightarrow R(\mcalA, \mcalL)^{1/p}[p^{1/p}] \hookrightarrow \cdots \hookrightarrow R(\mcalA, \mcalL)^{1/p^i}[p^{1/p^i}] \hookrightarrow \cdots
    \end{equation*}
    arising from \((R(\mcalA, \mcalL), (p))\) and the generic rank of transition maps are \(p[K(A)^{1/p}:K(A)]\) which is \(p\)-power.
    Its tilt is
    \begin{equation*}
        R(A, L)[|T|] \hookrightarrow R(A, L)^{1/p}[|T^{1/p}|] \hookrightarrow \cdots \hookrightarrow R(A, L)^{1/p^i}[|T^{1/p^i}|] \hookrightarrow \cdots.
    \end{equation*}
    In particular, we have an inequality
    \begin{equation*}
        \abs{H^i(\Spec(R(\mcalA, \mcalL))_{\et}, \setZ/\ell^n\setZ)} \leq \abs{H^i(\Spec(R(A, L)[|T|])_{\et}, \setZ/\ell^n\setZ)}
    \end{equation*}
    by \cite[Proposition 4.7]{ishiro2025Perfectoid}.
\end{example}

\subsection{Examples from affine semigroups}

The next example is generated from an affine semigroup ring following \Cref{ExplicitPerfectoidTowerDeltaRing}. This is based on examples of local log-regular rings in \Cref{LogRegularExample}.
This construction gives a lot of examples of perfectoid towers arising from integral domains and therefore this is one way to construct perfectoid towers easily (but this is closely related to local log-regular rings).

\begin{proposition}[{e.g., \cite[Definition 7.1]{miller2005Combinatorial}}] \label{ExampleSemigroup}
    Let \(\mathbf{a}_1, \dots, \mathbf{a}_r\) be a set of elements of \(\setZ_{\geq 0}^n\) for some \(n > 0\).
    Let \(H\) be a submonoid of \(\setZ_{\geq 0}^n\) generated by \(\mathbf{a}_1, \dots, \mathbf{a}_r\).
    Then the affine semigroup ring \(\setZ_p[H]\) is a \(\setZ_p\)-subalgebra of a polynomial ring \(\setZ_p[t_1, \dots, t_n]\) which is generated by \(\mathbf{t}^{\mathbf{a}_1}, \dots, \mathbf{t}^{\mathbf{a}_r}\) as a \(\setZ_p\)-algebra.
    This is a \(p\)-torsion-free finitely generated \(\setZ_p\)-algebra and the formula \(\mathbf{t}^h \mapsto \mathbf{t}^{ph}\) extends to a Frobenius-lift of the \(\setZ_p\)-algebra \(\setZ_p[H]\).
    Then the \((p, \mathbf{t}^{\mathbf{a}_1}, \dots, \mathbf{t}^{\mathbf{a}_r})\)-adic completion\footnote{See \cite[Lemma 8.15]{miller2005Combinatorial} for an explicit representation of \(\setZ_p[|H|]\).} \(\setZ_p[|H|]\) of \(\setZ_p[H]\) satisfies the assumption of \Cref{ExplicitPerfectoidTowerDeltaRing}. So we have a perfectoid tower
    \begin{equation*}
        \setZ_p[|H|] \hookrightarrow \setZ_p[p^{1/p}][|H^{1/p}|] \hookrightarrow \cdots \hookrightarrow \setZ_p[p^{1/p^i}][|H^{1/p^i}|] \hookrightarrow \cdots
    \end{equation*}
    arising from \((\setZ_p[|H|], (p))\) where \(H^{1/p^i}\) is the submonoid of \((1/p^i \cdot \setZ_{\geq 0})^n\) generated by \(1/p^i \cdot \mathbf{a}_1, \dots, 1/p^i \cdot \mathbf{a}_r\).
    Its tilt is
    \begin{equation*}
        \setF_p[|H|][|T|] \hookrightarrow \setF_p[|H^{1/p}|][|T|] \hookrightarrow \cdots \hookrightarrow \setF_p[|H^{1/p^i}|][|T|] \hookrightarrow \cdots
    \end{equation*}
    where \(T\) is a new variable.
\end{proposition}

To analyze some ring properties of affine semigroup rings over \(\setZ_p[H]\) not only over a field, we record the following lemma.

\begin{lemma} \label{SemigroupRingProperties}
    Let \(H\) be a submonoid of \(\setZ_{\geq 0}^n\) generated by \(\mathbf{a}_1, \dots, \mathbf{a}_r\) for some \(n > 0\).
    Assume that \(H\) is simplicial, namely, the convex rational polyhedral cone spanned by \(\mathbf{a}_1, \dots, \mathbf{a}_r\) in \(\setQ^n\) is spanned by \(r\)-elements in \(H\) where \(r\) is \(\rank_{\setZ} H \setZ^n\).
    \begin{enumerate}
        \item If \(\setF_{\ell}[H]\) is Cohen-Macaulay (resp., Gorenstein) for a prime \(\ell\), then \(\setZ_p[H]\) and \(\setZ_p[|H|]\) are Cohen-Macaulay (resp., Gorenstein) for any prime \(p\).
        \item If \(\setF_{\ell}[H]\) is normal for a prime \(\ell\), then \(\setZ_p[H]\) and \(\setZ_p[|H|]\) are normal and Cohen-Macaulay for any prime \(p\). Moreover, \(\setZ_p[|H|]\) with \(H \hookrightarrow \setZ_p[|H|]\) is a local log-regular ring.
    \end{enumerate}
    In particular, such properties can be tested in the case of \(\ell = 2\) only.
\end{lemma}

\begin{proof}
    (1): Since Cohen-Macaulayness and Gorensteinness are stable under the quotient by a non-zero-divisor and the completion, it is enough to check such properties for the affine semigroup ring \(\setF_p[H]\) over \(\setF_p\).
    By \cite[Theorem (1) and (2)]{goto1976Affine}, such properties only depend on the simplicial monoid \(H\) and this proves (1).

    (2): By \cite[Proposition 7.25]{miller2005Combinatorial}, the normality of an affine semigroup ring \(k[H]\) over a field \(k\) depends only on the monoid \(H\). So the assumption implies that \(\setQ_p[H]\) is normal (in other words, the monoid \(H\) is normal).
    We know that the polynomial ring \(\setZ_p[\underline{t}]\) is integrally closed in \(\setQ_p[\underline{t}]\) and the equation \(\setZ_p[H] = \setZ_p[\underline{t}] \cap \setQ_p[H]\) holds.
    The normality of \(\setQ_p[H]\) leads to the normality of \(\setZ_p[H]\).
    Moreover, in this case, the normality of \(\setQ_p[H]\) is stable under the completion due to \cite[Th\'eor\`eme 2]{zariski1950Normalite}. So a similar argument shows that \(\setZ_p[|H|]\) is normal.

    Using \cite[Theorem 1]{hochster1972Rings} (see also \cite[Corollary 13.43]{miller2005Combinatorial}), the normality of \(H\) implies that \(\setZ_p[H]\) and \(\setZ_p[|H|]\) are Cohen-Macaulay.
    The log-regularity follows from the definition of local log-regular rings: Normal affine semigroup \(H\) is fine, sharp, and saturated.
\end{proof}

By using this \Cref{SemigroupRingProperties} and a computer algebra system such as \cite{Macaulay2}, we can give the following examples: The first one is non-Cohen-Macaulay and the second one is Cohen-Macaulay but not normal.

\begin{example} \label{ExampleSemigroupnonCM}
    Take an affine semigroup ring \(\setZ_p[s, st, st^3, st^4]\) whose Frobenius lift is induced from \(s \mapsto s^p\), \(t \mapsto t^p\).
    This is a non-Cohen-Macaulay integral domain of dimension \(3\) for any prime \(p\) and so is the completion \(\setZ_p[|s, st, st^3, st^4|]\) with respect to the prime ideal \((s, st, st^3, st^4)\). By \Cref{ExampleSemigroup}, we have a perfectoid tower
    \begin{align*}
        \setZ_p[|s, st, st^3, st^4|] & \hookrightarrow \setZ_p[p^{1/p}][|s^{1/p}, s^{1/p}t^{1/p}, s^{1/p}t^{3/p}, s^{1/p}t^{4/p}|] \hookrightarrow \\
        & \cdots \hookrightarrow \setZ_p[p^{1/p^i}][|s^{1/p^i}, s^{1/p^i}t^{1/p^i}, s^{1/p^i}t^{3/p^i}, s^{1/p^i}t^{4/p^i}|] \hookrightarrow \cdots
    \end{align*}
    arising from \((\setZ_p[|s, st, st^3, st^4|], (p))\) whose first term \(\setZ_p[|s, st, st^3, st^4|]\) is a non-Cohen-Macaulay and non-normal complete local domain of dimension \(3\).
\end{example}

\begin{example} \label{ExampleSemigroupnonNor}
    Take an affine semigroup ring \(\setZ_p[s^2, s^3]\) whose Frobenius lift is induced from \(s \mapsto s^p\).
    This is a Cohen-Macaulay non-normal domain of dimension \(2\) for any prime \(p\) and so is the completion \(\setZ_p[|s^2, s^3|]\) with respect to the prime ideal \((s^2, s^3)\). By \Cref{ExampleSemigroup}, we have a perfectoid tower
    \begin{equation*}
        \setZ_p[|s^2, s^3|] \hookrightarrow \setZ_p[p^{1/p}][|s^{2/p}, s^{3/p}|] \hookrightarrow \cdots \hookrightarrow \setZ_p[p^{1/p^i}][|s^{2/p^i}, s^{3/p^i}|] \hookrightarrow \cdots
    \end{equation*}
    arising from \((\setZ_p[|s^2, s^3|], (p))\) whose first term \(\setZ_p[|s^2, s^3|]\) is a non-regular non-normal local complete intersection domain of dimension \(2\).
\end{example}


\subsection{Examples from \(\delta\)-stable ideals}

Next, we give examples of Noetherian perfectoid towers arising from \(\delta\)-rings by using \(\delta\)-stable ideals.

First one is a tower of rings generated from a \(\delta\)-stable ideal which generalizes examples arising from regular and log-regular rings (\Cref{RegularExample} and \Cref{LogRegularExample}). Especially, the completion of any Stanley--Reisner ring has a perfectoid tower. This is based on an example (\Cref{SquareFreeExample}) taught to the author by Shinnosuke Ishiro.
One advantage of this construction is that the conditions are easy to check by hand (or using computer algebra systems) and therefore this gives a practical way to construct (a little bit complicated) perfectoid towers.

\begin{proposition}[{\(\delta\)-stable ideals}] \label{ExampleDeltaStabilization}
    Let \(k\) be a perfect field of characteristic \(p\) and set \(W \defeq W(k)\)\footnote{Even if \(W = \setZ\), the conclusions (1) and (2) also hold for the \((p, \underline{T})\)-adic completion \(\setZ_p[|\underline{T}|]/J\).} with the unique Frobenius lift.
    Let \(W[\underline{T}] \defeq W[T_1, \dots, T_n]\) be a polynomial ring over \(W\) and let \(J\) be an ideal of \(W[\underline{T}]\) which is contained in \((p, T_1, \dots, T_n)\).
    Assume that there exists a \(\delta\)-structure on \(W[\underline{T}]\) such that compatible with the Frobenius lift on \(W\) and \(\delta(J) \subseteq J\).\footnote{This \(\delta\)-structure is not necessarily defined as \(\delta(T_i) = 0\). If \(W[\underline{T}]/J\) is \(p\)-torsion-free, the existence of such a \(\delta\)-structure on \(W[\underline{T}]\) is equivalent to the existence of a Frobenius lift on \(W[\underline{T}]/J\). So this relates the Frobenius liftability of a singularity in positive characteristic as mentioned in \Cref{PerfectoidTowerSingularity}.}
    Then \(W[\underline{T}]/J\) itself, its \(T\)-adic completion \(W[|\underline{T}|]/J\), and its localization \(W[\underline{T}]_{(p, \underline{T})}/J\) with respect to the maximal ideal \((p, \underline{T}) \subseteq W[\underline{T}]\) have unique \(\delta\)-\(W[\underline{T}]\)-algebra structures (\cite[Lemma 2.9 and Lemma 2.15]{bhatt2022Prismsa}).
    Take a distinguished element \(d\) in the \(\delta\)-ring \(W[\underline{T}]\) and fix a generator \(J = (f_1, \dots, f_r)\). Then we have the following.
    \begin{enumerate}
        \item If \(p, d\) is a regular sequence on \(W[\underline{T}]/J\) and the \(p\)-th power map \(W[\underline{T}]/(J, p, d) \xrightarrow{a \mapsto a^p} W[\underline{T}]/(J, p, d^p)\) is injective, then there exists a perfectoid tower arising from \((W[|\underline{T}|]/(J, d), (p))\) with injective transition maps and its tilt is isomorphic to the perfect tower arising from \((k[|\underline{T}|]/J, (p))\).
        \item If \(W[\underline{T}]/J\) is \(p\)-torsion-free and \(W[\underline{T}]/(p, J)\) is reduced, then there exists a perfectoid tower arising from \((W[|\underline{T}|]/J, (p))\) (resp., \((W[\underline{T}]_{(p, \underline{T})}/J, (p))\)) with injective transition maps and their tilts are isomorphic to the perfect tower arising from \((k[|\underline{T}, T'|]/J, (p))\) where \(T'\) is a new variable. If moreover \(W[|\underline{T}|]/J\) (resp., \(W[\underline{T}]_{(p, \underline{T})}/J\)) is a normal domain and \(W[|\underline{T}|]/(p, J)\) (resp., \(W[\underline{T}]_{(\underline{T})}/(p, J)\)) is an integral domain, then the generic extension of those transition maps have \(p\)-power degree.
    \end{enumerate}
\end{proposition}

\begin{proof}
    In (1), we have a bounded prism \((W[|\underline{T}|]/J, (d))\). Taking the \((\underline{T})\)-adic completion, \(p, d\) is also a regular sequence on \(W[|\underline{T}|]/J\). To apply \Cref{PerfectoidTowerPrism}, we need to show that \(W[|\underline{T}|]/(p, J)\) is \(p\)-root closed in \(W[|\underline{T}|]/(p, J)[1/f]\).
    As in the proof of \Cref{PerfectoidTowerPrism}, this is equivalent to showing that the \(p\)-th power map \(W[|\underline{T}|]/(J, p, d) \xrightarrow{a \mapsto a^p} W[|\underline{T}|]/(J, p, d^p)\) is injective. This follows from the assumption and the faithfully flat property of the \((\underline{T})\)-adic completion.

    Next, we show (2). As above, \(W[|\underline{T}|]/J\) (resp., \(W[\underline{T}]_{(p, \underline{T})}/J\)) is \(p\)-torsion-free by taking the \((\underline{T})\)-adic completion (resp., \((p)\)-Zariskization). By using analytically unramified property of finitely generated \(k\)-algebras, \(W[|\underline{T}|]/(p, J)\) is reduced (see, for example, \cite[Theorem 4.6.3, Proposition 9.1.3, and Theorem 9.2.2]{swanson2006Integral}).
    Also, the \(p\)-Zariskization \((1+(p))^{-1}W[\underline{T}]/(p, J) \cong W[\underline{T}]/(p, J)\) is reduced.
    By \Cref{ExplicitPerfectoidTowerDeltaRing}, we have a perfectoid tower arising from \((W[|\underline{T}|]/J, (p))\) (resp., \((W[\underline{T}]_{(p, \underline{T})}/J, (p))\)) with injective transition maps.
    If we assume that \(W[|\underline{T}|]/J\) (resp., \(W[\underline{T}]_{(p, \underline{T})}/J\)) is a normal domain and \(W[|\underline{T}|]/(p, J)\) (resp., \(W[\underline{T}]_{(p, \underline{T})}/(p, J)\)) is an integral domain, then the generic extension of those transition maps have \(p\)-power degree.
\end{proof}

\begin{corollary} \label{CorExampleDeltaStabilization}
    Keep the notation of \Cref{ExampleDeltaStabilization}.
    If \(\delta(T_i) = 0\), then we also have the following.
    \begin{enumerate}
        \item[(1')] The distinguished element \(d\) is written by \(p-f\) for some \(f \in (T_1, \dots, T_n)\). If \(p, f\) is a regular sequence on \(W[\underline{T}]/J\) and the \(p\)-th power map \(W[\underline{T}]/(J, p, f) \xrightarrow{a \mapsto a^p} W[\underline{T}]/(J, p, f^p)\) is injective, then the tower
        \begin{equation}
            W[|\underline{T}|]/(f_1, \dots, f_r, p-f) \hookrightarrow \cdots \hookrightarrow W[|\underline{T}^{1/p^i}|]/(f_1^{1/p^i}, \dots, f_r^{1/p^i}, p-f) \hookrightarrow \label{EqPerfdTowerDeltaStablePrism}
        \end{equation}
        is a perfectoid tower arising from \((W[|\underline{T}|]/(J, p-f), (p))\) and its tilt is isomorphic to the perfect tower
        \begin{equation}
            k[|\underline{T}|]/(f_1, \dots, f_r) \hookrightarrow \cdots \hookrightarrow k[|\underline{T}^{1/p^i}|]/(f_1^{1/p^i}, \dots, f_r^{1/p^i}) \hookrightarrow \cdots \label{TiltDeltStablePrism}
        \end{equation}
        arising from \((k[|\underline{T}|]/J, (p))\).
        \item[(2')] If \(W[\underline{T}]/J\) is \(p\)-torsion-free and \(W[\underline{T}]/(p, J)\) is reduced, then the towers
        \begin{align}
            W[|\underline{T}|]/(f_1, \dots, f_r) & \hookrightarrow \cdots \hookrightarrow W[|\underline{T}^{1/p^i}|][p^{1/p^i}]/(f_1^{1/p^i}, \dots, f_r^{1/p^i}) \hookrightarrow \cdots~\text{and} \label{EqPerfdTowerDeltaStable} \\
            W[\underline{T}]_{(p, \underline{T})}/(f_1, \dots, f_r) & \hookrightarrow \cdots \hookrightarrow W[\underline{T}^{1/p^i}]_{(p, \underline{T}^{1/p^i})}[p^{1/p^i}]/(f_1^{1/p^i}, \dots, f_r^{1/p^i}) \hookrightarrow \cdots \label{EqPerfdTowerDeltaStableZariski}
        \end{align}
        are perfectoid towers arising from \((W[|\underline{T}|]/J, (p))\) and \(W[\underline{T}]_{(p, \underline{T})}/(J, (p))\) respectively and their tilts are isomorphic to the perfect tower
        \begin{equation}
            k[|\underline{T}, T'|]/(f_1, \dots, f_r) \hookrightarrow \cdots \hookrightarrow k[|\underline{T}^{1/p^i}, {T'}^{1/p^i}|]/(f_1^{1/p^i}, \dots, f_r^{1/p^i}) \hookrightarrow \cdots \label{TiltDeltStable}
        \end{equation}
        arising from \((k[|\underline{T}, T'|]/J, (p))\) where \(T'\) is a new variable.
    \end{enumerate}
\end{corollary}

\begin{proof}
    In (1'), any distinguished element of a complete Noetherian local \(\delta\)-ring \(W[|\underline{T}|]/J\) with \(\delta(T_i) = 0\) can be written as \(p-f\) for some \(f \in (T_1, \dots, T_n)\) up to unit by the proof of \cite[Lemma 5.1]{ishizuka2026Prismatic}.
    Other assertions follow from (1) and (2) in \Cref{ExampleDeltaStabilization}.
\end{proof}

One typical example of \(\delta\)-stable ideals is the ideal generated by a square-free monomial.
First we give an example based on (\ref{EqPerfdTowerDeltaStablePrism}).

\begin{example} \label{SquareFreeExamplePrism}
    Set a \(\delta\)-structure on \(\setZ_p[|X, Y, Z, W|]\) by \(\delta(X) = \delta(Y) = \delta(Z) = \delta(W) = 0\).
    Take a \(\delta\)-stable ideal \((XY)\) in \(\setZ_p[|X, Y, Z, W|]\) and a distinguished element \(p - ZW\) as above (\Cref{CorExampleDeltaStabilization} (a')).
    Then \(p, ZW\) is a regular sequence on \(\setZ_p[X, Y, Z, W]/(XY)\) and the \(p\)-th power map \(\setF_p[X, Y, Z, W]/(XY, ZW) \xrightarrow{a \mapsto a^p} \setF_p[X, Y, Z, W]/(XY, Z^pW^p)\) is injective.
    Therefore, \Cref{CorExampleDeltaStabilization} (a') tells us that the tower
    \begin{align*}
        \setZ_p[|X, Y, Z, W|]/(XY, p-ZW) & \hookrightarrow \setZ_p[|X^{1/p}, Y^{1/p}, Z^{1/p}, W^{1/p}|]/(X^{1/p}Y^{1/p}, p-ZW) \hookrightarrow \\
        & \cdots \hookrightarrow \setZ_p[|X^{1/p^i}, Y^{1/p^i}, Z^{1/p^i}, W^{1/p^i}|]/(X^{1/p^i}Y^{1/p^i}, p-ZW) \hookrightarrow \cdots
    \end{align*}
    is a perfectoid tower arising from \((\setZ_p[|X, Y, Z, W|]/(XY, p-ZW), (p))\). The first term is a ramified complete intersection but not an integral domain.
\end{example}

Secondly, we give an example based on (\ref{EqPerfdTowerDeltaStable}).

\begin{example}[{Square-free monomial case}] \label{SquareFreeExample}
    Set a \(\delta\)-structure on \(\setZ_p[|X, Y, Z|]\) by \(\delta(X) = \delta(Y) = \delta(Z) = 0\) and take a \(\delta\)-stable ideal \((XY, YZ)\) in \(\setZ_p[|X, Y, Z|]\).
    Then we can show that the tower
    \begin{align*}
        \setZ_p[|X, Y, Z|]/(XY, YZ) & \hookrightarrow \setZ_p[|p^{1/p}, X^{1/p}, Y^{1/p}, Z^{1/p}|]/(X^{1/p}Y^{1/p}, Y^{1/p}Z^{1/p}) \hookrightarrow \\
        & \cdots \hookrightarrow \setZ_p[|p^{1/p^i}, X^{1/p^i}, Y^{1/p^i}, Z^{1/p^i}|]/(X^{1/p^i}Y^{1/p^i}, Y^{1/p^i}Z^{1/p^i}) \hookrightarrow \cdots
    \end{align*}
    is a perfectoid tower arising from \((\setZ_p[|X, Y, Z|]/(XY, YZ), (p))\) by \Cref{CorExampleDeltaStabilization} (b').
    This example gives a perfectoid tower whose first term \(\setZ_p[|X, Y, Z|]/(XY, YZ)\) is not Cohen-Macaulay. The same argument works for any quotient of \(\setZ_p[|X_1, \dots, X_n|]\) by square-free monomial ideals, i.e., the \((\underline{T})\)-completion of any Stanley--Reisner ring \(\setZ_p[\underline{T}]/I_{\Delta}\) over \(\setZ_p\) for any simplicial complex \(\Delta\) has a perfectoid tower arising from \((\setZ_p[|\underline{T}|]/I_{\Delta}, (p))\) (for the notion of Stanley--Reisner rings and simplicial complexes, see for example \cite{francisco2014Survey}).
\end{example}

\subsection{Examples from \(\delta\)-stabilization of ideals}

Any ideal \(J\) can be extended to the \(\delta\)-stabilization \(J_{\delta}\) of \(J\), the universal \(\delta\)-stable ideal containing \(J\), as in \cite[Example 2.10]{bhatt2022Prismsa}. At least \(p = 2\) or \(p = 3\), we give more complicated examples (\Cref{p23CaseDeltaStabilization}) than above by using the \(\delta\)-stabilization of ideals.
For convenience, we define the notion of \(\delta\)-height of an ideal of \(\delta\)-rings.

\begin{definition} \label{DeltaHeight}
    Let \(R\) be a \(\delta\)-ring and let \(J\) be an ideal of \(R\).
    The \emph{\(\delta\)-height} of \(J\) is defined as
    \begin{align*}
        \height^\delta(J) \defeq \inf\set{n \geq 0}{\delta^{n+1}(J) \subseteq (\cup_{0 \leq i \leq n} \delta^i(J))R},
    \end{align*}
    where \((\cup_{0 \leq i \leq n} \delta^i(J))R\) is the ideal of \(R\) generated by \(\cup_{0 \leq i \leq n} \delta^i(J)\).
    We assume that the value is \(\infty\) if the sequence of ideals \(\{(\cup_{0 \leq i \leq n} \delta^i(J))R\}_{n \geq 0}\) does not stabilize although if \(R\) is Noetherian, this value is finite.
    If it is finite, then the \(\delta\)-stabilization \(I_{\delta}\) is the same as the ideal generated by \(\cup_{0 \leq i \leq \height^\delta(J)} \delta^i(J)\).
    We denote by \(\height^\delta(f)\) for an element \(f\) of \(R\) the \(\delta\)-height of the ideal \((f)\).
\end{definition}

In the principal ideal case, there is an upper bound of the \(\delta\)-height of an element.

\begin{lemma} \label{DeltaHeightMonomial}
    Let \(R\) be a \(\delta\)-ring and let \(f\) be an element of \(R\).
    Take a \(\phi\)-monomial decomposition \(f = \sum_{i=1}^m k_iM_i\) of \(f\) with \(k_i \in \setZ\) and \(\phi\)-monomials \(M_i \in R\) in the sense of \cite[Definition 3.12]{kawakami2022Fedder}, namely, \(\varphi(M_i) = M_i^p\).\footnote{For example, any element \(f\) of a polynomial ring \(\setZ[X_1, \dots, X_n]\) has a \(\phi\)-monomial decomposition \(f = \sum_{i=1}^m k_iM_i\) with \(\phi\)-monomials \(M_i\) because any polynomial over \(\setZ\) can be written as a sum of monomials with coefficients \(1\) or \(-1\).}
    Then we have \(\height^\delta(f) \leq \height^\delta(k_1t_1 + \cdots + k_mt_m)\) where \(t_i \in \setZ[t_1, \dots, t_m]\) is the variable of the \(\delta\)-ring \(\setZ[t_1, \dots, t_m]\) with \(\delta(t_i) = 0\).
\end{lemma}

\begin{proof}
    Set \(\overline{f} \defeq k_1t_1 + \cdots + k_mt_m\) in \(\setZ[t_1, \dots, t_m]\).
    Since \(\delta(M_i) = 0\) holds for any \(\phi\)-monomials \(M_i\), we can take a map of \(\delta\)-rings \(\setZ[t_1, \dots, t_m] \to R\) sending \(t_i\) to \(M_i\).
    If \(\delta^{n+1}(\overline{f})\) belongs to the ideal \((\overline{f}, \delta(\overline{f}), \dots, \delta^n(\overline{f}))\setZ[t_1, \dots, t_m]\), then we have \(\delta^{n+1}(f) \in (f, \delta(f), \dots, \delta^n(f))R\) and thus \(\height^\delta(f) \leq n\).
\end{proof}

By using a computer algebra system such as Macaulay2 \cite{Macaulay2}, we can compute the \(\delta\)-stabilization of a given ideal in a polynomial ring and check the sufficient conditions in \Cref{CorExampleDeltaStabilization}.
Based on observations of computer calculations, we give a general construction of perfectoid towers from \(\delta\)-stabilization in \Cref{p23CaseDeltaStabilization} when \(p = 2, 3\). Hereafter, we will prepare its proof.

\begin{lemma} \label{DeltaBeta}
    Let \(m\) be an integer greater than or equal to \(3\).
    Set a \(\delta\)-structure on \(\setZ_p[|X_1, \dots, X_m|]\) by \(\delta(X_i) = 0\).
    Take \(f \defeq X_1^{n_1} + \cdots + X_m^{n_m}\) for \(n_i \geq 1\).
    \begin{enumerate}
        \item In \(\setZ_p[X_1, \dots, X_m]/(f)\), the element \(\delta(f)\) is the same as a non-zero polynomial
        \begin{equation} \label{EqDeltaBeta}
            \beta^{(m)} = \beta^{(m)}(X_2^{n_2}, \dots, X_m^{n_m}) = \beta^{(m)}_p X_2^{n_2p} + \beta^{(m)}_{p-1} X_2^{n_2(p-1)} + \beta^{(m)}_{p-2} X_2^{n_2(p-2)} + \cdots + \beta^{(m)}_0
        \end{equation}
        for some polynomials \(\beta^{(m)}_i = \beta^{(m)}_i(X_2^{n_2}, \dots, X_m^{n_m})\) in \(\setZ_p[X_3, \dots, X_m]\).
        \item Explicitly, \(\beta^{(m)}_p = (1+(-1)^p)/p\) and \(\beta^{(m)}_{p-1} = (-1)^{p}(X_3^{n_3} + \cdots + X_m^{n_m})\) hold. The last term \(\beta^{(m)}_0 = \beta^{(m)}_0(X_2^{n_2}, \dots, X_m^{n_m})\) is the same as \(\beta^{(m-1)}(X_3^{n_3}, \dots, X_m^{n_m})\) for any \(m \geq 4\).
        \item Let \(\Lambda\) be \(\setZ_p\) or \(\setF_p\) and let \(g = g_1f + g_2\beta^{(m)}\) be a polynomial in \(\Lambda[X_1, \dots, X_m]\) with \(g_1, g_2 \in \Lambda[X_1, \dots, X_m]\). If the degree of \(X_1\) in \(g\) is strictly less than \(n_1\), then \(g_1\) is contained in the ideal \((\beta^{(m)})\) in \(\Lambda[X_1, \dots, X_m]\).
    \end{enumerate}
\end{lemma}

\begin{proof}
    Under modulo \(f = X_1^{n_1} + \cdots + X_m^{n_m}\), we have
    \begin{align}
        \delta(f) & \equiv -\sum_{\substack{0 \leq e_1, \dots, e_m \leq p-1 \\ e_1 + \cdots + e_m = p}} \frac{(p-1)!}{e_1! \cdots e_m!}(-1)^{e_1}(X_2^{n_2} + \cdots + X_m^{n_m})^{e_1} X_2^{n_2e_2} \cdots X_m^{n_me_m} \nonumber \\
        & = -\sum_{\substack{0 \leq f_2, \dots, f_m, e_2, \dots, e_m \leq p-1 \\ f_2 + \cdots + f_m + e_2 \cdots + e_m = p \\ f_2 + \cdots + f_m \leq p-1}} \frac{(p-1)!}{f_2! \cdots f_m! e_2! \cdots e_m!}(-1)^{f_2 + \cdots + f_m}X_2^{n_2(e_2+f_2)} \cdots X_m^{n_m(e_m+f_m)} \label{EqDeltaf1}
    \end{align}
    which is an element of \(\setZ[X_2, \dots, X_m]\).
    We will compute the coefficient of \(X_2^{pn_2}\) and \(X_2^{n_2(p-1)}\) in the summation (\ref{EqDeltaf1}).

    In the case of \(e_2+f_2 = p\), the limitation of the sum in (\ref{EqDeltaf1}) implies that \(e_3 = \cdots = e_m = 0\), \(f_3 = \cdots = f_m = 0\), and \(e_2 > 0\). So the coefficient of \(X_2^{n_2p}\) is
    \begin{align*}
        \sum_{e_2=1}^{p-1} \frac{(p-1)!}{(p-e_2)!e_2!}(-1)^{p-e_2} & = \frac{1}{p} \parenlr{\sum_{e_2=0}^p \frac{p!}{(p-e_2)!e_2!}(-1)^{p-e_2} - \parenlr{(-1)^{p-0} + (-1)^{p-p}}} \\
        & = -\frac{(-1)^p+1}{p}
    \end{align*}
    and this is \(-1\) if \(p = 2\) and \(0\) if \(p > 2\).

    Next one is \(e_2+f_2 = p-1\) and we can assume that \(e_i+f_i = 1\) for some \(i > 2\). Then \(e_j = f_j = 0\) for \(j \neq 2, i\) and there are two cases: \(e_i = 0\) and \(f_i = 1\), or \(e_i = 1\) and \(f_i = 0\).
    In the former case, the coefficient of \(X_2^{n_2(p-1)}X_i^{n_i}\) is
    \begin{align*}
        \sum_{e_2=1}^{p-1} \frac{(p-1)!}{(p-1-e_2)!1!e_2!}(-1)^{(p-1-e_2)+1} & = -\parenlr{\sum_{e_2=0}^{p-1} \frac{(p-1)!}{(p-1-e_2)!e_2!}(-1)^{p-1-e_2} - (-1)^{p-1}} \\
        & = (-1)^{p-1}.
    \end{align*}
    In the latter case, the coefficient of \(X_2^{n_2(p-1)}X_i^{n_i}\) is
    \begin{align*}
        \sum_{e_2=0}^{p-1} \frac{(p-1)!}{(p-1-e_2)!e_2!1!}(-1)^{p-1-e_2} = 0. 
    \end{align*}
    Therefore, the terms in (\ref{EqDeltaf1}) of degree \(n_2(p-1)\) in \(X_2\) can be summed up as \((-1)^{p}X_2^{n_2(p-1)}(X_3^{n_3} + \cdots + X_m^{n_m})\).
    Consequently, the summation (\ref{EqDeltaf1}) can be written as
    \begin{equation}
        \beta^{(m)} \defeq \frac{1+(-1)^{p}}{p}X_2^{n_2p} + (-1)^{p}X_2^{n_2(p-1)}(X_3^{n_3} + \cdots + X_m^{n_m}) + \beta^{(m)}_{p-2} X_2^{n_2(p-2)} + \cdots + \beta^{(m)}_0
    \end{equation}
    for some polynomials \(\beta_i^{(m)}\) in \(X_3, \dots, X_m\) and set \(\beta^{(m)}_p \defeq (1+(-1)^{p})/p\) and \(\beta^{(m)}_{p-1} \defeq (-1)^{p}(X_3^{n_3} + \cdots + X_m^{n_m})\). In particular, \(\beta^{(m)}_{p-1}\) is a non-zero element and so is \(\beta^{(m)}\).
    The last term \(\beta^{(m)}_0\) is the case that \(e_2+f_2 = 0\) in (\ref{EqDeltaf1}) and this is the same as \(\beta^{(m-1)}(X_3^{n_3}, \dots, X_{m}^{n_m})\) for any \(m \geq 4\).
    This shows the assertions (1) and (2).

    We show the next assertion (3). Let \(g = g_1f + g_2\beta^{(m)}\) be a polynomial in \(\Lambda[X_1, \dots, X_m]\) with \(g_1, g_2 \in \Lambda[X_1, \dots, X_m]\) and the degree of \(X_1\) in \(g\) is strictly less than \(n_1\).
    Write the polynomials \(g_1 = a_NX_1^N + a_{N-1}X_1^{N-1} + \cdots + a_1X_1 + a_0\) and \(g_2 = b_MX_1^M + b_{M-1}X_1^{M-1} + \cdots + b_1X_1 + b_0\) such that \(a_i\) and \(b_j\) are in \(\Lambda[X_2, \dots, X_m]\) and \(a_N\) and \(b_M\) are non-zero.
    Since \(g\) has the degree of \(X_1\) strictly less than \(n_1\), we can deduce that \(N+n_1 = M\).
    For simplicity, set \(\alpha \defeq X_2^{n_2} + \cdots + X_m^{n_m}\) and \(\beta \defeq \beta^{(m)}\).
    So we have \(g = g_1 (X_1^{n_1} + \alpha) + g_2 \beta\), namely,
    \begin{align}
        g & = (a_NX_1^{N+n_1} + \cdots + a_{N-n_1}X_1^{N-n_1+n_1} + \cdots + a_0X_1^{n_1}) + (\alpha a_N X_1^N + \cdots + \alpha a_0) + \label{EqInitialIdeal} \\
        & + (\beta b_MX_1^M  + \cdots + \beta b_NX_1^{N} + \cdots + \beta b_0). \nonumber
    \end{align}
    It suffices to show that \(a_i\) are divisible by \(\beta\) in \(\Lambda[X_2, \dots, X_m]\) for all \(0 \leq i \leq N\).
    If \(n_1 > N\), the above equation shows
    \begin{equation} \label{EqInitialIdeal2}
        a_iX_1^{i+n_1} + \beta b_{i+n_1}X_1^{i+n_1} = 0
    \end{equation}
    for all \(0 \leq i \leq N\)
    since the degree of \(X_1\) in \(g\) is strictly less than \(n_1\). This implies that \(a_i\) is divisible by \(\beta\) for all \(0 \leq i \leq N\).

    If \(n_1 \leq N\), the same vanishing (\ref{EqInitialIdeal2}) holds for any \(N-n_1+1 \leq i \leq N\) and thus \(a_i\) is divisible by \(\beta\) for all \(N-n_1+1 \leq i \leq N\).
    The equation (\ref{EqInitialIdeal}) gives
    \begin{equation} \label{EqInitialIdeal3}
        a_{i-n_1}X_1^i + \alpha a_iX_1^i + \beta b_iX_1^i = 0
    \end{equation}
    for any \(n_1 \leq i \leq N\) because of \(n_1 \leq N\).
    Whether \(N - n_1 + 1\) is smaller than or greater than \(n_1\), setting \(a_i\) as \(0\) for \(i < 0\), we have the same equation (\ref{EqInitialIdeal3}) for any \(N - n_1 + 1 \leq i \leq N\).
    Then the divisibility of \(a_i\) shows that \(a_{i-n_1}\) is divisible by \(\beta\) for any \(N-n_1+1 \leq i \leq N\), i.e., \(a_i\) is divisible by \(\beta\) for any \(N-2n_1+1 \leq i \leq N-n_1\).
    Using (\ref{EqInitialIdeal3}) again, \(a_{i-n_1}\) is divisible by \(\beta\) for any \(N-2n_1+1 \leq i \leq N-n_1\), i.e., \(a_i\) is divisible by \(\beta\) for any \(N-3n_1+1 \leq i \leq N-2n_1\).
    Repeating this argument, \(a_i\) is divisible by \(\beta\) for all \(0 \leq i \leq N\).
    Consequently, in both cases, \(g_1\) is contained in \((\beta)\) in \(\Lambda[X_1, \dots, X_m]\).
\end{proof}

\begin{example} \label{ExampleBeta}
    In \Cref{DeltaStabilizationComputer} later, we use \(\beta^{(p+1)} = \beta^{(p+1)}(X_2^{n_2}, \dots, X_{p+1}^{n_{p+1}})\) for prime \(p\). We give explicit representations of \(\beta^{(p+1)}\) for \(p = 2\) and \(p = 3\):
    \begin{align*}
        \beta^{(2+1)} & = X_2^{2n_2} + X_2^{n_2}X_3^{n_3} + X_3^{2n_3}, \\
        \beta^{(3+1)} & = (X_2^{n_2} + X_3^{n_3})(X_3^{n_3} + X_4^{n_4})(X_4^{n_4} + X_2^{n_2}). \\
    \end{align*}
    This calculation follows from \(\beta^{(p+1)}(t_2, \dots, t_{p+1})\) and substitution \(t_i = X_i^{n_i}\) for \(i = 2, \dots, p+1\) as in \Cref{DeltaHeightMonomial}.
\end{example}

\begin{corollary} \label{DeltaHeightandInitialIdeal}
    Keep the notation in \Cref{DeltaBeta}.
    \begin{enumerate}
        \item The sequence \(p, f, \delta(f)\) is a regular sequence on \(\setZ_p[|X_1, \dots, X_m|]\) and especially \(\height^\delta(f) \geq 1\).
        \item The initial ideal of \((f, \delta(f))\) in \(\Lambda[X_1, \dots, X_m]\) with respect to the lexicographic order \(X_1 > X_2 > \cdots > X_m\) is \((X_1^{n_1}, X_2^{2n_2})\) if \(p = 2\) and \((X_1^{n_1}, X_2^{n_2(p-1)}X_3)\) if \(p > 2\).
    \end{enumerate}
\end{corollary}

\begin{proof}
    For the first statement, it is enough to show that \(f\) is a non-zero-divisor on \(\setF_p[X_1, \dots, X_m]/(\beta^{(m)})\) since \(\delta(f)\) and \(\beta^{(m)}\) are equivalent under modulo \(f\) by \Cref{DeltaBeta}(1).
    If a polynomial \(g_1\) in \(\setF_p[X_1, \dots, X_m]\) satisfies \(g_1f \in (\beta^{(m)})\), then \(g_1\) is contained in \((\beta^{(m)})\) by using \Cref{DeltaBeta}(3).

    For the initial ideal, take \(g \in (f, \delta(f)) = (f, \beta^{(m)})\). Dividing by \(f\), we may assume that the degree of \(X_1\) in \(g\) is strictly less than \(n_1\). By \Cref{DeltaBeta}(3) again, \(g\) is contained in \((\beta^{(m)})\) in \(\Lambda[X_1, \dots, X_m]\) and thus its initial term is generated by the initial term of \(\beta^{(m)}\), which is \(X_2^{2n_2}\) if \(p = 2\) and \(X_2^{n_2(p-1)}X_3\) if \(p > 2\).
\end{proof}


\begin{lemma} \label{DeltaStabilizationComputer}
    Set a \(\delta\)-structure on \(\setZ_p[|X_1, \dots,  X_{p+1}|]\) by \(\delta(X_i) = 0\).
    Take \(f \defeq X_1^{n_1} + \cdots + X_{p+1}^{n_{p+1}}\) and \(\beta \defeq \beta^{(p+1)}(X_2^{n_2}, \dots, X_{p+1}^{n_{p+1}})\) in \(\setZ_p[X_1, \dots, X_{p+1}]\) for \(n_i \geq 1\).
    If \(n_1\) and \(n_2\) are prime to \(p\) and \(\setF_p[X_2, \dots, X_{p+1}]/(\beta)\) is reduced, then so is \(\setF_p[|X_1, \dots, X_{p+1}|]/(f, \delta(f))\).
\end{lemma}

\begin{proof}
    Write \(\alpha \defeq X_2^{n_2} + \cdots + X_m^{n_m}\) and \(\beta \defeq \beta^{(p+1)}(X_2^{n_2}, \dots, X_{p+1}^{n_{p+1}})\).
    We first show that \(\setF_p[X_1, \dots, X_{p+1}]/(f, \delta(f))\) is reduced.
    It is enough to show that \((X_1^{n_1}+\alpha, \beta) \cap \setF_p[X_1^p, \dots, X_{p+1}^p]\) is contained in \(((X_1^{n_1}+\alpha)^p, \beta^p)\) in \(\setF_p[X_1, \dots, X_{p+1}]\). Take \(g \in (X_1^{n_1}+\alpha, \beta) \cap \setF_p[X_1^p, \dots, X_{p+1}^p]\). Dividing by \((X_1^{n_1}+\alpha)^p\) in \(\setF_p[X_1^p, \dots, X_{p+1}^p]\), we may assume that the degree of \(X_1\) in \(g\) is strictly less than \(pn_1\).
    
    It is enough to show that \(g\) is contained in \((\beta^p)\) in \(\setF_p[X_1, \dots, X_{p+1}]\).
    Since \(g\) belongs to \((X_1^{n_1}+\alpha, \beta)\), dividing \(g\) by \(X_1^{n_1} + \alpha\) and \(\beta\) and \Cref{DeltaHeightandInitialIdeal}(2) yield a representation \(g = g_1 (X_1^{n_1} + \alpha) + g_2 \beta\) where \(g_1 = a_NX_1^N + a_{N-1}X_1^{N-1} + \cdots + a_1X_1 + a_0\) and \(g_2 = b_MX_1^M + b_{M-1}X_1^{M-1} + \cdots + b_1X_1 + b_0\) are polynomials such that \(a_i\) and \(b_j\) are in \(\setF_p[X_2, \dots, X_{p+1}]\), \(a_N\) and \(b_M\) are non-zero, and \(N < n_1(p-1)\) and \(M < n_1\) hold.

    It is enough to show that \(a_i \in (\beta^p)\) and \(b_i \in (\beta^{p-1})\).
    Considering the same equation (\ref{EqInitialIdeal}) and the condition \(g \in \setF_p[X_1^p, \dots, X_{p+1}^p]\), the initial degree \(N+n_1\) is in \(p\setZ\) and then \(N\) is prime to \(p\) since \(n_1\) is so. Similarly, because of \(M < n_1\), the term \(a_kX_1^{k+n_1}\) is zero if \(k+n_1\) is prime to \(p\) and is contained in \(\setF_p[X_1^p, \dots, X_{p+1}^p]\) if \(k+n_1\) is divisible by \(p\) for any \(n_1 \leq k+n_1 \leq N+n_1\). Namely, we have
    \begin{equation*}
        a_k =
        \begin{cases}
        0 & \text{if } k \not\equiv -n_1 \pmod{p}, \\
        \in \setF_p[X_2^p, \dots, X_{p+1}^p] & \text{if } k \equiv -n_1 \pmod{p}
        \end{cases}
    \end{equation*}
    for any \(0 \leq k \leq N\).
    If \(N > M\), the term \(\alpha a_N X_1^N\) is the only term whose degree of \(X_1\) is \(N\) which is prime to \(p\) and it should be zero, but this contradicts \(\alpha a_N \neq 0\). So we have \(n_1 > M \geq N\) and the same argument for \(b_M\) shows that the only cases are \(M = N\) or \(M \in p\setZ\). In both case, similarly as above, considering \(\alpha a_k X_1^k + \beta b_k X_1^k\) and \(\beta b_k X_1^k\), we have
    \begin{align*}
        \alpha a_k + \beta b_k =
        \begin{cases}
        0 & \text{if } k \not\equiv 0 \pmod{p} \text{ and } 0 \leq k \leq N, \\
        \in \setF_p[X_2^p, \dots, X_{p+1}^p] & \text{if } k \equiv 0 \pmod{p} \text{ and } 0 \leq k \leq N
        \end{cases} \\
        \beta b_k =
        \begin{cases}
        0 & \text{if } k \not\equiv 0 \pmod{p} \text{ and } N < k \leq M, \\
        \in \setF_p[X_2^p, \dots, X_{p+1}^p] & \text{if } k \equiv 0 \pmod{p} \text{ and } N < k \leq M.
        \end{cases}
    \end{align*}
    Clearly, if \(N < k \leq M\), then we have \(\beta b_k \in \setF_p[X_2^p, \dots, X_{p+1}^p]\).
    If \(k \not\equiv -n_1 \bmod p\) and \(0 \leq k \leq N\), then \(\beta b_k \in \setF_p[X_2^p, \dots, X_{p+1}^p]\) also holds.
    
    Our assumption says that \(\setF_p[X_2, \dots, X_{p+1}]/(\beta)\) is reduced and this implies the containment \((\beta) \cap \setF_p[X_2^p, \dots, X_{p+1}^p] \subseteq (\beta^p)\).
    Therefore, any \(k\) such that \(k \not\equiv -n_1 \bmod p\) and \(0 \leq k \leq M\) satisfies \(a_k = 0 \in (\beta^p)\) and \(b_k \in (\beta^{p-1})\).
    If \(k \equiv -n_1 \bmod p\), then \(a_k\) belongs to \(\setF_p[X_2^p, \dots, X_{p+1}^p]\) and \(\alpha a_k + \beta b_k = 0\). Using \Cref{DivisibleClaim} below, we can deduce that \(a \in (\beta) \cap \setF_p[X_2^p, \dots, X_{p+1}^p]\) and thus \(a \in (\beta^p)\). This also implies \(b_k \in (\beta^{p-1})\).
    Consequently, \(g\) is contained in \((\beta^p)\) in \(\setF_p[X_1, \dots, X_{p+1}]\) and this shows the reduced property of \(\setF_p[X_1, \dots, X_{p+1}]/(f)_{\delta}\) and this implies the reduced property of \(\setF_p[|X_1, \dots, X_{p+1}|]/(f)_{\delta}\) as in the proof of \Cref{ExampleDeltaStabilization} (2).

    \begin{claim} \label{DivisibleClaim}
        Set \(\alpha = X_2^{n_2} + \cdots + X_{p+1}^{n_{p+1}}\) and \(\beta = \beta^{(p+1)}(X_2^{n_2}, \dots, X_{p+1}^{n_{p+1}})\) in \(\setF_p[X_2, \dots, X_{p+1}]\) such that \(n_2\) is prime to \(p\).
        If \(a\) and \(b\) in \(\setF_p[X_2, \dots, X_{p+1}]\) satisfies \(\alpha a + \beta b = 0\), then \(a\) is divisible by \(\beta\) and \(b\) is divisible by \(\alpha\) in \(\setF_p[X_2, \dots, X_{p+1}]\).
    \end{claim}

    \begin{claimproof}
        Since \(n_2\) is prime to \(p\), the polynomial \(\alpha = X_2^{n_2} + X_3^{n_3} + \cdots + X_{p+1}^{n_{p+1}}\) in \(\setF_p(X_3, \dots, X_{p+1})[X_2]\) is separable.
        Each root \(r_1, \dots, r_{n_2}\) of this polynomial satisfies \(r_i^{n_2} + X_3^{n_3} + \cdots + X_{p+1}^{n_{p+1}} = 0\) and then \(\restr{\beta}{X_2=r_i} \cdot \restr{b}{X_2=r_i} = 0\) in \(\setF_p[X_3, \dots, X_{p+1}]\).
        By the construction of \(\beta\) in \Cref{DeltaBeta}(1), the polynomial \(\restr{\beta}{X_2=r_i}\) is
        \begin{align*}
            & -\sum_{\substack{0 \leq e_1, \dots, e_{p+1} \leq p-1 \\ e_1 + \cdots + e_{p+1} = p}} \frac{(p-1)!}{e_1! \cdots e_{p+1}!}(-1)^{e_1}(r_i^{n_2} + X_3^{n_3} \cdots + X_{p+1}^{n_{p+1}})^{e_1} r_i^{n_2e_2} X_3^{n_3e_3} \cdots X_{p+1}^{n_{p+1}e_{p+1}} \\
            & = -\sum_{\substack{0 \leq e_2, \dots, e_{p+1} \leq p-1 \\ e_2 + \cdots + e_{p+1} = p}} \frac{(p-1)!}{e_2! \cdots e_{p+1}!}(-1)^{e_2}(X_3^{n_3} \cdots + X_{p+1}^{n_{p+1}})^{e_2}X_3^{n_3e_3} \cdots X_{p+1}^{n_{p+1}e_{p+1}}
        \end{align*}
        and this is the same as the non-zero element \(\beta^{(p)}(X_3^{n_3}, \dots, X_{p+1}^{n_{p+1}})\) by the definition of \(\beta^{(p)}\) in \Cref{DeltaBeta}.
        This implies that \(\restr{b}{X_2=r_i} = 0\) for any \(i\) and thus \(b\) is divisible by \(\alpha = \prod_{i=1}^{n_2}(X_2-r_i)\) in \(\overline{\setF_p(X_3, \dots, X_{p+1})}[X_2]\) where \(\overline{\setF_p(X_3, \dots, X_{p+1})}\) is the algebraic closure of \(\setF_p(X_3, \dots, X_{p+1})\).
        Take an element \(b' = b'_MX_2^M + \cdots + b'_0\) in \(\overline{\setF_p(X_3, \dots, X_{p+1})}[X_2]\) such that \(b = \alpha b' = (X_2^{n_2} + a')b'\) for \(a' \in \setF_2[X_3, \dots, X_{p+1}]\).
        Since \(b\) is in \(\setF_p[X_2, \dots, X_{p+1}]\), comparing the coefficients shows that each \(b'_i\) belongs to \(\setF_p[X_3, \dots, X_{p+1}]\) and then \(b\) is divisible by \(\alpha\) in \(\setF_p[X_2, \dots, X_{p+1}]\). This also implies \(a\) is divisible by \(\beta\) in \(\setF_p[X_2, \dots, X_{p+1}]\).
    \end{claimproof}

\end{proof}

We give an example of non-monomial \(\delta\)-stable ideals. This tower arises from a complete intersection but not a log-regular ring unlike \Cref{LogRegularExample}.

\begin{corollary} \label{p23CaseDeltaStabilization}
    Let \(p\) be \(2\) or \(3\) and let \(f\) be an element \(X_1^{n_1} + X_2^{n_2} + X_3^{n_3}\) in \(\setZ_2[X_1, X_2, X_3]\) or \(X_1^{n_1} + X_2^{n_2} + X_3^{n_3} + X_4^{n_4}\) in \(\setZ_3[X_1, X_2, X_3, X_4]\) for \(n_i \geq 1\) respectively.
    Then we have the following.
    \begin{enumerate}
        \item The \(\delta\)-height \(\height^{\delta}(f)\) is \(1\) and the sequence \(p, f, \delta(f)\) is a regular sequence on \(\setZ_p[|X_1, \dots, X_{p+1}|]\).
        \item The mod \(p\) reduction \(\setF_p[|X_1, \dots, X_{p+1}|]/(f)_{\delta}\) is reduced if and only if at least two of \((n_1, \dots, n_{p+1})\) are prime to \(p\).
        \item If at least two of \((n_1, n_2, n_3)\) are odd and \(X_2^{2n_2} + X_2^{n_2}X_3^{n_3} + X_3^{2n_3}\) is an irreducible polynomial\footnote{This condition does not hold in general, e.g., \(n_1 = n_2 = n_3 = 5\).} in \(\setZ_2[X_2, X_3]\) (resp., in \(\setF_2[X_2, X_3]\)), then the quotients \(\setZ_2[X_1, X_2, X_3]/(f)_{\delta}\) (resp., \(\setF_2[X_1, X_2, X_3]/(f)_{\delta}\)) is an integral domain.
    \end{enumerate}
    In particular, when at least two of \((n_1, \dots, n_{p+1})\) are prime to \(p\), we have a perfectoid tower arising from \((\setZ_p[|X_1, \dots, X_{p+1}|]/(f)_{\delta}, (p))\) as (\ref{EqPerfdTowerDeltaStable}) and the first term \(\setZ_p[|X_1, \dots, X_{p+1}|]/(f)_{\delta}\) is a \(p\)-torsion-free complete intersection of dimension \(p\) by (1). If \(p = 2\), it is not isomorphic to any local log-regular ring.
\end{corollary}

\begin{proof}
    The last claim follows from (1), (2), and \Cref{NonLogRegular} below since \(\setZ_2[|X_1, X_2, X_3|]/(f)_{\delta}\) is a \(2\)-dimensional non-regular unramified Gorenstein complete local ring and its mod \(2\)-reduction is reduced if at least two of \(n_1\), \(n_2\), or \(n_3\) are odd.
    Note that the non-regularity follows from the following argument: Take the maximal ideal \(\mfrakm \defeq (2, X_1, X_2, X_3)\) of \(\setZ_2[|X_1, X_2, X_3|]\).
    If any \(n_i\) is strictly greater than \(1\), then \((f)_{\delta}\) is contained in \(\mfrakm^2\) and thus the embedding dimension of \(\setZ_2[|X_1, X_2, X_3|]/(f)_{\delta}\) is \(4\).
    If some \(n_i\) is equal to \(1\), for example \(n_1 = 1\), then \(\setZ_2[|X_1, X_2, X_3|]/(f)_{\delta}\) is isomorphic to \(\setZ_2[|X_2, X_3|]/(X_2^{2n_2} + X_2^{n_2}X_3^{n_3} + X_3^{2n_3})\) but still the embedding dimension is \(3\) since the defining ideal is contained in \((2, X_2, X_3)^2\). Therefore, in any case, \(\setZ_2[|X_1, X_2, X_3|]/(f)_{\delta}\) is not regular.

    (1): Combining \Cref{DeltaHeightandInitialIdeal} and \Cref{DeltaHeightMonomial}, it is enough to show that \(\delta^2(t_1+t_2+t_3)\) (resp., \(\delta^2(t_1+t_2+t_3+t_4)\)) is contained in the ideal \((t_1+t_2+t_3, \delta(t_1+t_2+t_3))\) (resp., \((t_1+t_2+t_3+t_4, \delta(t_1+t_2+t_3+t_4))\)).
    This follows from a direct computation based on \Cref{ExampleBeta}.
    In particular, \((f)_{\delta} = (f, \beta^{(p+1)}(X_2^{n_2}, \dots, X_{p+1}^{n_{p+1}}))\) by \Cref{DeltaBeta}(1).

    (2): Let us prove the only if part. We prove the contrapositive, namely, we assume that at most one of \((n_1, \dots, n_{p+1})\) is prime to \(p\) and show that \(\setF_p[|X_1, \dots, X_{p+1}|]/(f)_{\delta}\) is not reduced.
    Since the ideal \((f)_{\delta} = (f, \beta^{(p+1)}(X_2^{n_2}, \dots, X_{p+1}^{n_{p+1}}))\) is independent on the order of \(X_1, \dots, X_{p+1}\), we can change the indices so that \(n_2\), \(n_3\) (resp., \(n_2\), \(n_3\), and \(n_4\)) are divisible by \(p\).
    
    First assume that \(p=2\) and \(n_2\) and \(n_3\) are even. Set \(n_2 = 2n_2'\) and \(n_3 = 2n'_3\) for some positive integer \(n_2', n_3'\). Then we have
    \begin{equation*}
        (X_2^{n_2} + X_2^{n_2'}X_3^{n_3'} + X_3^{n_3})^2 = X_2^{2n_2} + X_2^{n_2}X_3^{n_3} +X_3^{2n_3} = 0
    \end{equation*}
    in \(\setF_p[|X_1, X_2, X_3|]/(f)_{\delta}\). By \Cref{DeltaHeightandInitialIdeal}(2), this element \(X_2^{n_2} + X_2^{n_2'}X_3^{n_3'} + X_3^{n_3}\) is non-zero in \(\setF_p[|X_1, X_2, X_3|]/(f)_{\delta}\) because of the initial degree \(n_2 < 2n_2\).
    Therefore, \(\setF_p[|X_1, X_2, X_3|]/(f)_{\delta}\) is not reduced.

    Next we prove when \(p=3\). Assume that \(n_2\), \(n_3\) and \(n_4\) are divisible by \(3\). Set \(n_2 = 3n_2'\), \(n_3 = 3n'_3\) and \(n_4 = 3n_4'\) for some positive integer \(n_2', n_3', n_4'\). Based on \Cref{ExampleBeta}, we have
    \begin{equation*}
        ((X_2^{n_2'} + X_3^{n_3'})(X_2^{n_2'} + X_4^{n_4'})(X_4^{n_4'} + X_3^{n_3'}))^3 = 0
    \end{equation*}
    in \(\setF_3[|X_1, X_2, X_3, X_4|]/(f)_{\delta}\). This element \((X_2^{n_2'} + X_3^{n_3'})(X_2^{n_2'} + X_4^{n_4'})(X_4^{n_4'} + X_3^{n_3'})\) is non-zero because of the initial degree \(2n_2'\) of \(X_2\) is strictly less than the degree \(2n_2\) of \(X_2\) in the initial ideal in \Cref{DeltaHeightandInitialIdeal} and thus \(\setF_3[|X_1, X_2, X_3, X_4|]/(f)_{\delta}\) is not reduced.

    Next, we prove if part, namely, we assume that at least two of \((n_1, \dots, n_{p+1})\) are prime to \(p\) and show that \(\setF_p[|X_1, \dots, X_{p+1}|]/(f)_{\delta}\) is reduced.
    Since \((f)_{\delta}\) does not change by changing the indices of \(X_1, \dots, X_{p+1}\), we may assume that \(n_1, n_2\) are prime to \(2\) (resp., \(3\)).
    By (1) together with \Cref{DeltaStabilizationComputer} and \Cref{ExampleBeta}, it is enough to show that
    \begin{align*}
        \setF_2[X_2, X_3]/(X_2^{2n_2} + X_2^{n_2}X_3^{n_3} + X_3^{2n_3}) \quad \text{and} \\
        \setF_3[X_2, X_3, X_4]/((X_2^{n_2} + X_3^{n_3})(X_2^{n_3} + X_4^{n_4})(X_4^{n_4} + X_3^{n_3}))
    \end{align*}
    are reduced. This follows if we show \((\beta) \cap \setF_p[X_2^p, \dots, X_{p+1}^p] \subseteq (\beta^p)\) in \(\setF_p[X_2, \dots, X_{p+1}]\).

    In case of \(p = 2\), we already assume that \(n_1\) and \(n_2\) are odd. Take \(c = \beta b \in (\beta) \cap \setF_2[X_2^2, X_3^2]\). Dividing \(b\) by \(\beta = X_2^{2n_2} + X_2^{n_2}X_3^{n_3} + X_3^{2n_3}\) in \(\setF_2[X_3][X_2]\), we can write \(b = \beta b' + r\) for some \(b'\) and \(r\) in \(\setF_2[X_2, X_3]\) where the degree of \(X_2\) in \(r\) is strictly less than \(2n_2\).
    Suppose that \(X_2\) has an odd degree in \(b'\) and let \(K\) be the largest odd degree of \(X_2\) in \(b'\) with a non-zero coefficient.
    Then the leading term of \(c = \beta^2 b' + \beta r\) whose degree of \(X_2\) is \(K + 4n_2\) is non-zero because of the leading term of \(\beta^2 b'\). This contradicts the assumption that \(c\) belongs to \(\setF_2[X_2^2, X_3^2]\).
    So \(c - \beta^2 b' = \beta r \in (\beta) \cap \setF_2[X_2^2, X_3]\) holds.
    Therefore, replacing \(b\) by \(r\), we have \(c = \beta b \in (\beta) \cap \setF_2[X_2^2, X_3]\) and the degree of \(X_2\) in \(b\) is strictly less than \(2n_2\).
    It is enough to show that \(b = 0\).
    Assume the converse, namely, we write \(b = b_MX_2^M + \cdots + b_1X_2 + b_0\) where \(b_i \in \setF_2[X_3]\), \(b_M \neq 0\), and \(M < 2n_2\).
    By \Cref{ExampleBeta}, we have \(\beta = \beta^{(2+1)} = X_2^{2n_2} + X_2^{n_2}X_3^{n_3} + X_3^{2n_3}\) and thus
    \begin{align*}
        c & = \beta b = (b_MX_2^{M+2n_2} + \cdots + b_0X_2^{2n_2}) + (b_MX_2^{M+n_2}X_3^{n_3} + \cdots + b_{n_2}X_2^{2n_2}X_3^{n_3} + \cdots + b_0X_2^{n_2}X_3^{n_3}) \\
        & + (b_MX_2^MX_3^{2n_3} + \cdots + b_{n_2}X_2^{n_2}X_3^{2n_3} + \cdots + b_0X_3^{2n_3}).
    \end{align*}
    Since this belongs to \(\setF_2[X_2^2, X_3]\), the leading degree \(M+2n_2\) is even because of \(b_M \neq 0\) and thus \(M\) is even.
    If \(M < n_2\), then \(b_MX_2^{M+n_2}X_3^{n_3}\) is the only non-zero term whose degree of \(X_2\) is \(M+n_2\). This number \(M+n_2\) becomes odd and thus \(b_M\) should vanish but this contradicts the assumption.
    If \(M \geq n_2\), then
    \begin{equation} \label{EqOddDegreeVanishing}
        b_kX_2^{k+2n_2} + b_{k+n_2}X_2^{k+2n_2}X_3^{n_3} = 0
    \end{equation}
    holds for any odd number \(k\) in \(0 \leq k \leq M-n_2\) since we assume \(M < 2n_2\).
    Because of \(b_kX_2^kX_3^{2n_3} = 0\) for any odd number \(k\) in \(0 \leq k \leq n_2-1\), the equality (\ref{EqOddDegreeVanishing}) implies \(b_{k+n_2} = 0\) for any odd number \(k\) in \(0 \leq k \leq \min\{M-n_2, n_2-1\}\), i.e., \(b_l = 0\) for any even number \(l\) in \(n_2 \leq l \leq \min\{M, 2n_2-1\}\).
    Since \(M\) is even and strictly less than \(2n_2\), this vanishing implies \(b_M = 0\) but this contradicts \(b_M \neq 0\).
    Therefore, \(b = 0\) and this completes the proof.

    For the case of \(p=3\), we already assume that \(n_1\) and \(n_2\) are prime to \(3\). Take \(c = \beta b \in (\beta) \cap \setF_3[X_2^3, X_3^3, X_4^3]\).
    Since \(X_2^{n_2} + X_3^{n_3}\) and \(X_2^{n_2} + X_4^{n_4}\) are separable in \(\setF_3(X_3, X_4)[X_2]\), so is their product \(\beta\).
    Taking the different roots \(s_1, \dots, s_{2n_2}\) of \(\beta\) in \(\overline{\setF_3(X_3, X_4)}\), we have \(\restr{c}{X_2=s_i} = 0\).
    Since \(c\) belongs to \(\setF_3[X_2^3, X_3^3, X_4^3]\) and \(s_i\) are different from each other, \(c\) is divisible by \((X_2-s_i)^3\) and thus \(c\) is divisible by \(\beta^3\) in \(\overline{\setF_3(X_3, X_4)}[X_2]\).
    Therefore, since \(c\) and \(\beta\) are both in \(\setF_3[X_3, X_4][X_2]\), as in the proof of \Cref{DivisibleClaim}, comparing the coefficients shows that \(c\) is divisible by \(\beta^3\) in \(\setF_3[X_2, X_3, X_4]\).

    (3): Set \(\Lambda\) to be \(\setZ_2\) or \(\setF_2\) and \(f_2 \defeq X_2^{2n_2} + X_2^{n_2}X_3^{n_3} + X_3^{2n_3}\). Note that \((f)_{\delta} = (f, f_2)\) in \(\Lambda[X_1, X_2, X_3]\) by (1).
    If \(n_2\) and \(n_3\) are odd numbers, then \(X_2^{n_2} + X_3^{n_3}\) has no factors with multiplicity strictly greater than \(1\) in \(\Lambda[X_2, X_3]\) (this follows from, for example, taking partial derivation) and Eisenstein criterion implies that \(X_1^{n_1} + X_2^{n_2} + X_3^{n_3}\) is irreducible in \(\Lambda[X_1, X_2, X_3]\). Other cases also follow from the same argument and so \(f\) is a monic irreducible polynomial in \(\Lambda[X_1, X_2, X_3]\).
    Then we can take a finite free extension \(\Lambda[X_2, X_3] \hookrightarrow \Lambda[X_1, X_2, X_3]/(f)\) and its base change 
    \begin{equation*}
        R \defeq \Lambda[X_2, X_3]/(f_2) \hookrightarrow \Lambda[X_1, X_2, X_3]/(f, f_2) \eqdef S
    \end{equation*}
    is also a finite free extension from an integral domain \(R\) since \(f_2\) is irreducible in \(\Lambda[X_2, X_3]\) by our assumption. Especially \(S\) is \(R\)-torsion-free and induces an injection \(S \hookrightarrow (R\setminus \{0\})^{-1}S\).
    Take any non-zero element \(s\) of \(S\) and we can find a monic polynomial \(s^m + r_1s^{m-1} + \cdots + r_m = 0\) with \(r_i \in R\).
    Then \(s (s^{m-1} + r_1s^{m-2} + \cdots + r_{m-1}) = -r_m\) in \(S\) holds and thus \(s\) is a unit in \((R \setminus \{0\})^{-1}S\).
    This implies that \((R \setminus \{0\})^{-1}S\) is a field and thus the subring \(S = \Lambda[X_1, X_2, X_3]/(f, f_2)\) is also an integral domain.
\end{proof}

\begin{example} \label{ExampleDeltaStabilizationp2}
    Set a \(\delta\)-structure on \(\setZ_2[X, Y, Z]\) by \(\delta(X) = \delta(Y) = \delta(Z) = 0\).
    Assume \(p = 2\) and take \(f \defeq X^3 + Y^4 + Z^5\) in \(\setZ_2[X, Y, Z]\). Then \Cref{DeltaStabilizationComputer} tells us that
    \begin{align*}
        (f)_{\delta} = (X^3 + Y^4 + Z^5, Y^8 + X^3Y^4 + X^6)
    \end{align*}
    holds in \(\setZ_2[X, Y, Z]\).
    By \Cref{p23CaseDeltaStabilization}, the \((X, Y, Z)\)-adic completion (resp., localization at \((p, X, Y, Z)\)) of the tower
    \begin{align*}
        & \setZ_2[X, Y, Z]/(X^3 + Y^4 + Z^5, Y^8 + X^3Y^4 + X^6) \hookrightarrow \\
        & \cdots \hookrightarrow \setZ_2[2^{1/2^i}][X^{1/2^i}, Y^{1/2^i}, Z^{1/2^i}]/(X^{3/2^i} + Y^{4/2^i} + Z^{5/2^i}, Y^{8/2^i} + X^{3/2^i}Y^{4/2^i} + X^{6/2^i}) \hookrightarrow \cdots
    \end{align*}
    is a perfectoid tower arising from \((\setZ_2[|X, Y, Z|]/(X^3 + Y^4 + Z^5, Y^8 + X^3Y^4 + X^6), (2))\) (resp., \((\setZ_2[X, Y, Z]_{(2, X, Y, Z)}/(X^3 + Y^4 + Z^5, Y^8 + X^3Y^4 + X^6), (2))\)). The first term of the tower is a complete intersection (resp., complete intersection domain) but they are not log-regular rings.
\end{example}

We provide a sufficient condition for the non-log-regularity of the first term of the tower in \Cref{p23CaseDeltaStabilization}. This proof is based on a private communication with Shinnosuke Ishiro.

\begin{lemma} \label{NonLogRegular}
    Let \(R\) be a complete Noetherian local ring of dimension \(2\) with residue field \(k\) of mixed characteristic \((0, p)\).
    Assume that \(R\) is unramified Gorenstein but not regular and \(R/pR\) is reduced.
    Then \(R\) is not isomorphic to any local log-regular ring.
\end{lemma}

\begin{proof}
    If \(R\) is a local log-regular ring, there exists a log structure \(\mcalQ \xrightarrow{\alpha} R\) from a fine, sharp, and saturated monoid \(\mcalQ\) such that \(R/I_{\alpha}\) is regular and \(2 = \dim(R) = \dim(R/I_{\alpha}) + \dim(\mcalQ)\) (see \cite[Definition 2.19]{ishiro2025Perfectoid}).
    Since \(R\) is not regular, \(I_{\alpha}\) is non-zero and thus \(\dim(\mcalQ) > 0\). By the structure theorem of local log-regular ring (\cite[Theorem 2.22]{ishiro2025Perfectoid}), \(R \cong \setZ_p[|\mcalQ \oplus \setN^r|]/(p-f)\) holds where \(r \defeq \dim(R/I_{\alpha})\) and \(f\) has no non-zero constant term.
    If \(\dim(\mcalQ) = 1\), then \(\mcalQ\) should be equal to \(\setN\). So \(R \cong \setZ_p[|\setN^2|]/(\theta)\) is regular but this contradicts the non-regularity of \(R\).
    So \(\dim(\mcalQ)\) should be \(2\). Since \(R\) is Gorenstein, \cite[Remark 2.2 and Corollary 4.11]{ishiro2024Local} shows that \(R\) is isomorphic to \(\setZ_p[|s^{n+1}, st, t^{n+1}|]/(p-f) \cong \setZ_p[|x,y,z|]/(xz - y^{n+1}, p-g)\) for some \(g \in (x, y, z)\) and \(n \geq 2\). Since \(R\) is unramified, \(g\) is \(x\) or \(z\) and thus \(R\) is isomorphic to \(\setZ_p[|s^{n+1}, st, t^{n+1}|]/(p-s^{n+1})\).
    Taking modulo \(p\), this should be a non-reduced ring \(\setF_p[|s^{n+1}, st, t^{n+1}|]/(s^{n+1})\) but this contradicts the reduced property of \(R/pR\).
\end{proof}

\end{document}